\documentclass[graybox]{svmult}

\usepackage{mathptmx}       
\usepackage{helvet}         
\usepackage{courier}        
\usepackage{type1cm}  
\usepackage{longtable}
\usepackage{hyperref}
\usepackage{makeidx}         
\usepackage{graphicx}        
\usepackage{multicol}        
\usepackage[bottom]{footmisc}

\makeindex             

\usepackage{amstext,amsfonts,amssymb,amscd,amsbsy,amsmath}
\usepackage[ruled,vlined]{algorithm2e}
\usepackage{algpseudocode}
\usepackage{tikz-cd}	
\usepackage{tikz}

\newcommand{\PP}{\mathbb{P}}

\newcommand{\RR}{\mathbb{R}}
\newcommand{\ZZ}{\mathbb{Z}}

\DeclareMathOperator{\Gr}{Gr}
\DeclareMathOperator{\Flag}{\mathcal{F}\hspace{-1.6pt}\ell}

\DeclareMathOperator{\trop}{trop}

\DeclareMathOperator{\init}{in}
\DeclareMathOperator{\mult}{mult}
\DeclareMathOperator{\Proj}{Proj}

\newcommand{\val}{\mathfrak{v}}

\begin{document}

\title*{Computing toric degenerations of flag varieties}
\author{Lara Bossinger, Sara Lamboglia, Kalina Mincheva, and Fatemeh Mohammadi}
\institute{Lara Bossinger \at University of Cologne, Mathematisches Institut,
Weyertal 86 - 90,
50931 Cologne,
Germany.\\ \email{lbossing@math.uni-koeln.de}
\and Sara Lamboglia\at University of Warwick, Mathematics Institute, 
Coventry, CV4 7AL, 
United Kingdom.\\ \email{S.Lamboglia@warwick.ac.uk}
\and Kalina Mincheva\at Yale University, Department of Mathematics, 10 Hillhouse Ave., New Haven, CT 06511, USA. \\
\email{ kalina.mincheva@yale.edu}
\and
Fatemeh Mohammadi\at  
University of Bristol, 
School of Mathematics,
Bristol,
BS8 1TW,
United Kingdom.\\ \email{fatemeh.mohammadi@bristol.ac.uk}
}

\maketitle
\abstract{
We compute toric degenerations arising from the tropicalization of the full flag varieties $\Flag_4$ and $\Flag_5$ embedded in a product of Grassmannians.
For $\Flag_4$ and $\Flag_5$ we compare toric degenerations arising from string polytopes and the FFLV polytope with those obtained from the tropicalization of the flag varieties.
We also present a general procedure to find toric degenerations  in the cases where the initial ideal arising from a cone of the tropicalization of a variety is not prime. 
}

\section{Introduction}\label{sec:1}
Consider the variety $\Flag_n$ of full flags $\{0\}= V_0\subset V_1\subset\cdots\subset V_{n-1}\subset V_n=\mathbb C^n$ of vector subspaces of $\mathbb C^n$ with ${\rm dim}_{\mathbb C}(V_i) = i$. The flag variety $\Flag_n$ is naturally embedded in a product of Grassmannians using the Pl\"ucker coordinates. 
We denote by $I_n$ the defining ideal of $\Flag_n$ with respect to this embedding. 
We produce toric degenerations of $\Flag_n$ as Gr\"obner degenerations coming from the initial ideals associated to the maximal cones of $\trop(\Flag_n)$. Moreover, we compare these with certain toric degenerations arising from representation theory.

We will consider $1$-parameter toric degenerations of $\Flag_n$. These are flat families 
$\varphi:\mathcal F\to \mathbb {A}^{\! 1}$, where the fiber over zero (also called \emph{special} fiber) is a toric variety and all other fibers are isomorphic to $\Flag_n$. 
Once we have such a  degeneration, some of the algebraic
invariants of $\Flag_n$ will be  the same for all fibers, hence the computation can be done on the toric fiber. In the case of a toric variety such invariants are easier to compute than in the case of a general variety. In fact, they have a nice combinatorial description. Moreover, toric degenerations 
connect different areas of mathematics, such as symplectic geometry, representation theory, and algebraic geometry.

\medskip

Let $X=V(I)$ be a projective variety and $\trop(X)$ be its tropicalization. The initial ideals associated to the top-dimensional cones of $\trop(X)$ are good candidates to give toric degenerations, see Lemma \ref{lem:multiplicity} (and \cite[Proposition~1.1]{KristinFatemeh} for a more general statement). For example, in the case of Grassmannians $\Gr(2,\mathbb C^n)$ each maximal cone of  $\trop(\Gr(2,\mathbb C^n))$ gives a toric degeneration, see \cite{Speyer,Witaszek,BFFHL}. However, this is not true for the Grassmannians $\Gr(3,\mathbb C^n)$. In \cite{KristinFatemeh} Mohammadi  and Shaw   identify which maximal cones of $\trop(\Gr(3,\mathbb C^n))$ produce such degenerations.

The following are our main results. More detailed formulations can be found in
Theorem~\ref{flag4}, Theorem~\ref{flag5}, and Proposition~\ref{prop:output}.
We will call a maximal cone $C$ of $\trop(X)$ \emph{prime} if $\init_C(I):=\init_{\bf{w}}(I)$ is prime, with $\bf{w}$ a vector in the relative interior of $C$.

\begin{theorem}\label{Intro:Flag4}
The tropical variety $\trop(\Flag_4)\subset \mathbb R^{14}/\mathbb R^3$ is a 6-dimensional fan with $78$ maximal 
cones. From prime cones we obtain four non-isomorphic toric degenerations.
After applying Procedure 1 we obtain at least two additional non-isomorphic toric degenerations from non-prime cones.
\end{theorem}

\begin{theorem}\label{Intro:Flag5}
The tropical variety $\trop(\Flag_5)\subset \mathbb R^{30}/\mathbb {R}^4$ is a 10-dimensional fan 
with $69780$ maximal cones. From prime cones we obtain 180 non-isomorphic toric degenerations.

\end{theorem}

Toric degenerations of flag varieties and Schubert varieties have been studied intensively in representation theory over the last two decades.  We refer the reader to \cite{FFL16} for a nice overview on this topic and to the references therein. 

The main motivation of this paper is to study the
flat degenerations of flag varieties into toric varieties arising from the tropicalization  and to compare these degenerations to those associated to \emph{string polytopes} and the \emph{Feigin-Fourier-Littelmann-Vinberg polytope (FFLV polytope)}.

\begin{theorem}\label{ts-comparison} 
For $\Flag_4$ there is at least one new toric degeneration arising from prime cones of $\trop(\Flag_4)$ in comparison to those obtained from string polytopes and the FFLV polytope.

For $\Flag_5$ there are at least 168 new toric degenerations arising from prime cones of $\trop(\Flag_5)$ in comparison to those obtained from string polytopes and the FFLV polytope.
\end{theorem}

Our work is closely related to the theory of Newton--Okounkov bodies. Let $\Bbbk$ be a not necessarily algebraically closed field and $X$ a projective variety. It is possible to associate to every  prime cone in $\trop(X)$ a valuation with a finite \emph{Khovanskii basis} $B$ on the homogeneous coordinate ring $\Bbbk[X]$, see \cite[Lemma 5.7]{KM16}. This is a set of elements of $\Bbbk[X]$, such that their valuations generate the value semigroup $S(\Bbbk[X], \val)$.
The convex hull of $S(\Bbbk[X],\val)\cup \{0\}$ is referred to as the \emph{Newton--Okounkov cone}. After intersecting this cone with a particular hyperplane one obtains a convex body, called the \emph{Newton--Okounkov body}. When a finite Khovanskii basis exists,  \cite[Theorem 1.1]{An13} states that there is  a flat degeneration of the variety $X$ into a toric variety whose normalization has as associated polytope the Newton--Okounkov body. In this case the Newton--Okounkov body is a  polytope. The toric polytopes obtained in Theorem~\ref{flag4}, Theorem~\ref{flag5}, and Proposition~\ref{prop:output} can be seen as Newton--Okounkov bodies for the valuations defined in \S6.

\medskip

The paper is structured as follows. In \S\ref{sec:background} we provide the necessary background. We study the tropicalization of the flag varieties $\Flag_n$ for $n=4,5$ and the induced toric 
degenerations in \S\ref{sec:3}.  
The solutions to \cite[Problem~5 on Grassmannians]{Sturmfels} and \cite[Problem~6 on Grassmannians]{Sturmfels} can be found in Theorem~\ref{flag4}.

In \S\ref{sec:rep} we recall the definition of string cones, string polytopes, and FFLV polytope for regular dominant integral weights. We compute for $\Flag_4$ and $\Flag_5$ all string polytopes for the weight $\rho$, which is the sum of all fundamental weights. 
Moreover, in \S\ref{string:weight} for every string cone we construct a weight vector ${\bf w}_{\underline w_0}$ contained in the tropicalization of the flag variety  in order to further explore the connection between these two different approaches. The construction is inspired by Caldero \cite{Cal02}.

In \S\ref{Algorithmic approach} we give an algorithmic approach to solving \cite[Problem 1]{KM16} for a subvariety $X$ of a toric variety $Y$ when   each cone in $\trop(X)$ has multiplicity one. Procedure~1 aims at computing a new embedding $X'$ of $X$ in case $\trop(X)$ has some non-prime cones. Once we have such an embedding, we explain how to get new toric degenerations of $X$. We apply the procedure to $\Flag_4$. Furthermore, we explain how to interpret the procedure in terms of finding valuations with finite  Khovanskii basis on the algebra given by the homogeneous coordinate ring of $X$.

\section{Preliminary notions}\label{sec:background}
In this section we recall the definition of \text{flag variety} and we introduce the necessary background in tropical geometry. In fact, the key ingredient in the study of Gr\"obner toric degenerations of $\Flag_n$ is the subfan of the  Gr\"obner fan of $I_n$ given by the \textit{tropicalization} of $\Flag_n$.

We mostly  refer to the approach described in \cite{M-S} and we encourage the reader to look there for a more thorough introduction.
\medskip

Let $\Bbbk$  be a field with char$(\Bbbk)=0$ and consider on it the trivial valuation. We are mainly interested in the case when $\Bbbk=\mathbb{C}$. 
\begin{definition}
A \emph{complete flag} in the vector space ${\Bbbk}^n$ is a chain
\[\mathcal{V}:\ \{0\}= V_0\subset V_1\subset\cdots\subset V_{n-1}\subset V_n=\Bbbk^n\]
of vector subspaces of $\Bbbk^n$ with ${\rm dim}_{\Bbbk }(V_i) = i$. 
\end{definition}
The set of all complete flags in $\Bbbk^n$ is denoted by $\Flag_n$ and it has an algebraic variety structure. More precisely, it is a subvariety of the product of  Grassmannians $\Gr(1,\Bbbk^n)\times \Gr(2,\Bbbk ^n)\times\cdots \times \Gr(n-1,\Bbbk^n)$. 

Composing with the Pl\"ucker embeddings of the Grassmannians, $\Flag_n$ becomes a subvariety of 
$\mathbb{P}^{{n \choose 1}-1}\times\mathbb{P}^{{n \choose 2}-1}\times \cdots\times\mathbb{P}^{{n \choose n-1}-1}$ 
and so we can ask for its defining ideal $I_n$.  
Each point in the flag variety can be represented by an $(n-1)\times n$-matrix $M=(x_{i,j})$ whose first $d$ rows generate $V_d$. Each $V_d$ corresponds to a point in a Grassmannian. Moreover, they satisfy the condition $V_d\subset V_{d+1}$ for $d=0,\ldots,n-1$. 
In order to compute the  ideal $I_n$  defining $\Flag_n$ in   $\mathbb{P}^{{n \choose 1}-1}\times\mathbb{P}^{{n \choose 2}-1}\times \cdots\times\mathbb{P}^{{n \choose n-1}-1}$ we have to translate the inclusions  $V_d\subset V_{d+1}$ into polynomial equations. We define the map
\[
\varphi_n:\  \Bbbk[p_J:\ \varnothing\neq J\subsetneq \{1,\ldots,n\}]\rightarrow \Bbbk[x_{i,j}:\ 1\leq i\leq n-1,\ 1\leq j\leq n]
\]
sending each Pl\"ucker variable
$p_J$ to the determinant of the submatrix of $M$ with row indices $1,\ldots,|J|$ and column indices in $J$. The ideal $I_n$ of $\Flag_n$ is the kernel of $\varphi_n$.
There is an action of $S_n\rtimes \mathbb Z_2$  on $\Flag_n$. The symmetric group acts by permuting the columns of $M$. The action of $\mathbb Z_2$ maps a complete flag $\mathcal V$ to its complement, which is defined to be
\[\mathcal{V^{\bot}}:\ \{0\}= V_n^{\bot} \subset V_{n-1}^{\bot}\subset\cdots\subset V_{1}^{\bot}\subset V_0^{\bot}=\Bbbk^n.\]

We will hence do computations up to $S_n \rtimes \mathbb Z_2$-symmetry. We are interested in finding toric degenerations. These are degenerations whose special fiber is defined by a  {\em toric} ideal, i.e. a binomial prime ideal not containing monomials. This toric ideal arises as   \textit{initial ideal} of $I_n$.

\begin{definition}
Let $f=\sum a_{\textbf u} x^{\textbf u}$ with $\textbf{u}\in \mathbb {Z}_{\ge 0}^{n+1}$ be a polynomial in $S=\Bbbk[x_0,\ldots,x_n]$. For each $\textbf w\in \mathbb{R}^{n+1}$ we  define its \textit{initial form } to be 
\[
\init_{\textbf w}(f)= \sum_{\textbf w\cdot \textbf u \text{ is minimal}} a_{\textbf u} x^{\textbf u}.
\]
If $I$ is an ideal in $S$, then its \textit{initial ideal} with respect to $\textbf w$ is 
\[
\init_{\textbf w}(I)=\langle \init _{\textbf w}(f) : f\in I\rangle.
\]
\end{definition}

An important geometric property of initial ideals is that there exists a flat family over $\mathbb{A}^{\!1}$ for which the  fiber over $0$ is isomorphic to  $V(\init_{\textbf w}(I))$ and all the other fibers are isomorphic to the variety $V(I)$. Here, if $J$ is a homogeneous ideal of $S$ then we define $V(J):=\Proj(S/J)$ where  the grading on $S$ and hence on $S/J$ comes from the ambient space which has $S$ as homogeneous coordinate ring.

Let $t$ be the coordinate in $\mathbb {A}^{\!1}$, then the flat family is given by the ideal
\[
\tilde{I}_t=\langle t^{-\min_{{\textbf u}}\{{\textbf w}\cdot {\textbf u}\}} f(t^{w_0}x_0,\ldots,t^{w_n}x_n):\ f=\sum a_{{\textbf u}}x^{{\textbf u}}\ {\rm in\ } I   \rangle\subset \Bbbk[t,x_0,\ldots,x_n].
\]

This family gives a flat degeneration of the variety $V(I)$ into the variety $V(\init_{\textbf w}(I))$ 
called the \textit{Gr\"obner degeneration}. In order to look for toric degenerations, we study the \textit{tropicalization} of $V(I)$.

\begin{definition}
Let $f=\sum a_{\textbf u}x^{\textbf u}$ be any polynomial in $S$. The \textit{tropicalization} of $f$ is the function $\trop(f):\mathbb R^{n+1}\to \mathbb R$ given by 
\[
\trop(f)(\textbf w)=\min\{\textbf w\cdot \textbf u : \textbf u\in \mathbb Z_{\ge 0}^{n+1} \text{ and } a_{\textbf u}\neq 0\}.
\]
\end{definition}

Let  $f=\sum a_{\textbf u}x^{\textbf u}$ be a homogeneous polynomial in $S$. Then if $\textbf w-\textbf {v}=m\cdot\textbf 1$, for some $\textbf{v},\textbf{w}\in\mathbb R^{n+1}$,  ${\bf 1}=(1,\ldots,1)\in\mathbb R^{n+1}$ and $m\in \RR$, we have that the minimum in $\trop(f)({\bf w})$ and $\trop(f)({\bf v})$
is achieved for the same ${\bf u}\in \mathbb Z_{\ge 0}^{n+1}$ such that $ a_{\textbf u}\neq 0$.

\begin{definition}
Let $f$ be a homogeneous polynomial in $S$ and $V(f)$ the associated hypersurface in $\mathbb P^n$. Then the \textit {tropical hypersurface} of $f$ is defined to be 
\[
\trop(V(f))=\left\{ \textbf w\in \mathbb R^{n+1}/\mathbb R {\bf 1}\cong \mathbb R^n :\ 
\begin{matrix}
\text{the minimum in }\trop(f)(\textbf w)\\
\text{is achieved at least twice}
\end{matrix}\right\} .
\]
Let $I$ be a homogeneous ideal in $S$. The \textit{tropicalization} of the variety $V(I)\subset \mathbb P^n$ is defined to be
\[
\trop(V(I))=\bigcap_{f\in I}\trop(V(f)).
\]
\end{definition}

For every ${\bf w}\in \trop(V(I))$, $\init_{\bf w}(I)$ does not contain any monomial (see proof of \cite[Theorem 3.2.3]{M-S}). 
If $V(I)$ is a $(d-1)$-dimensional irreducible projective variety, then $\trop(V(I))$  is the  support of a rational fan  given by the quotient by $\mathbb R{\bf 1}$ of a subfan $F$ of the Gr\"obner fan of $I$  (\cite[Theorem~3.3.5]{M-S}). The fan $F$ has dimension $d$, which is the Krull dimension of $S/I$.
It is possible to quotient  by $\mathbb R{\bf 1}$ because  $I$ is homogeneous and hence  $\init_{\textbf w}(I)=\init_{\textbf {v}}(I)$ for every $\textbf w-\textbf {v}=m\cdot\textbf 1$ with $\textbf{v},\textbf{w}\in\mathbb R^{n+1}$ and $m\in \mathbb{R}$.
If we consider this fan structure on $\trop(V(I))$ we have that vectors in the relative interior of a cone give rise to the same initial ideal and vectors in distinct relative cone interiors induce  distinct initial ideals. 
For this reason we denote by $\init_{C}(I)$ the initial ideal of $I$ with respect to any $\textbf {w}$ in the relative interior of $C$. 

The tropicalization of a variety $X$ is non-empty only if $X$ intersects the torus $T^n=(\Bbbk^*)^{n+1}/\Bbbk^*$ non-trivially. In fact, \textbf{$\trop(X)$} is technically the tropicalization of $X\cap T^n$. 

In the same way the tropicalization can be defined when $S$ is the \textit{total coordinate ring} (see \cite[page 207]{Cox2} for a definition) of
$\mathbb P^{k_1}\times\cdots\times \mathbb P^{k_s}$. The ring $S$ has a $\mathbb{Z}^s$-grading  given by $\deg:\mathbb Z^{n+1}\to \mathbb Z^s$. An ideal $I$ defining an irreducible  subvariety $V(I)$ of  $\mathbb P^{k_1}\times\cdots\times \mathbb P^{k_s}$ is homogeneous with respect to this grading. The tropicalization of $V(I)$ is contained in $\mathbb R^{k_1+\ldots+k_s+s}/H$, where $H$ is an $s$-dimensional linear space spanned by the rows of the matrix $D$ associated to $\deg$. Similarly to the projective case, if $V(I)$ is a $d$-dimensional irreducible subvariety of $\mathbb P^{k_1}\times\cdots\times \mathbb P^{k_s}$, then $\trop(V(I))$  is the support of a fan,  which is  the quotient by $H$ of a rational $(d+s)$-dimensional subfan $F$ of the Gr\"obner fan of $I$. Here the Krull dimension of $S/I$ is $d+s$.

\medskip 

In the following we will always consider $\trop(V(I))$ with the fan structure defined above. 

\begin{remark}
A detailed definition of the tropicalization of a  general  toric variety $X_{\Sigma}$ and of its subvarieties can be found in \cite[Chapter 6]{M-S}. Note that we only consider the  tropicalization of the intersection of $V(I)$ with the torus of $X_{\Sigma}$ while in \cite[Chapter 6]{M-S} they introduce a generalized version  of $\trop(V(I))$ which includes the tropicalization of  the intersection of $V(I)$ with each orbit of  $X_{\Sigma}$.
\end{remark}

\medskip

Another property of $\trop(V(I))$ is that any fan structure on it can be balanced assigning a positive integer weight to every maximal cell. We will not explain the notion of balancing in detail and we consider an adapted version of the multiplicity defined in \cite[Definition 3.4.3]{M-S}.

\begin{definition}
Let $I\subset S$ be a homogeneous ideal and $\Sigma$ be a fan  with support $|\Sigma|=|\trop(V(I))|$ such that for  every cone  $C$ of $\Sigma$ the ideal  $\init_{\textbf w}(I)$ is constant for $\textbf{w}$ in the relative interior of $C$. For a maximal dimensional cone $C\in \Sigma$ we define the \emph{multiplicity} as 
$
\mult(C)=\sum_P \mult(P,\init_{C}(I)).
$
Here the sum is taken over the minimal associated primes $P$ of $\init_{C}(I)$ that do not contain monomials (see \cite[\S 3]{Eisenbud} or \cite[\S 4.7]{Cox}).
\end{definition}

As we have seen, each cone of $\trop(V(I))$ corresponds to an initial ideal which contains no monomials. Moreover, we will see that the  good candidates for  toric degenerations are the initial ideals corresponding to the relative interior of the maximal cones. We say a maximal cone is \emph{prime} if the corresponding initial ideal is prime.

\begin{lemma}\label{lem:multiplicity}
Let $I\subset S$ be a homogeneous ideal and  $C$ a maximal cone of $\trop(V(I))$.
If $\init_{C}(I)$ is a toric ideal, i.e. binomial and prime, then $C$ has multiplicity one.
If $C$ has multiplicity one, then $\init_{C}(I)$ has a unique toric ideal in its primary decomposition.
\end{lemma}
\begin{proof}
We first prove the lemma for $S$ the homogeneous coordinate ring of $\mathbb P^n$. Let $I'=\init_{C}(I) \Bbbk[x_1^{\pm 1},\ldots,x_n^{\pm 1}]$ and consider the subvariety $V(I')$ of the torus $T^n$. Then by \cite[Remark 3.4.4]{M-S} the multiplicity of a maximal cone $C$ is counting the number of $d$-dimensional torus orbits whose union is $V(I')$. If $\init_{C}(I)$ is toric, then $V(I')$ is an irreducible toric variety having a unique $d$-dimensional torus orbit. Hence $C$ has multiplicity one.

Suppose now $C$ has multiplicity one. This implies that $\init_C(I)$ contains one associated prime $J$, which does not contain monomials. The ideal $J$ has to be binomial since it is the ideal of the unique $d$-dimensional torus orbit contained in $V(I')$.

When $S$ is the total coordinate ring of $\mathbb P^{k_1}\times\cdots\times \mathbb P^{k_s}$, the torus is given by $T^{k_1}\times \cdots \times T^{k_s}\cong T^{k_1+\cdots +k_s}$. We may assume that for each $i$, 
\[
T^{k_i}=\{[1:a_1:\ldots:a_{k_i}]\in \mathbb P^{k_i}:\ a_j\neq 0\text{\ for\ all\ } j\}.
\]
The variables for $\mathbb P^{k_i}$ are denoted by  $x_{i,0},\dots,x_{i,k_i}$ for each $i$. We fix  the  Laurent polynomial ring 
\[
S'=\Bbbk[{x^{\pm 1}_{1,1}},\dots,x^{\pm 1}_{1,k_1},x^{\pm 1}_{2,1},\dots,x^{\pm 1}_{2,k_2},\dots,x^{\pm 1}_{s,1},\dots,x^{\pm 1}_{s,k_s}].
\]
We consider the ideal $I'=\init_{C}(I)S'$  in  $S'$ and the variety $V(I')$ as a subvariety of $T^{k_1+\ldots +k_s}$. Then the proof proceeds as before.\qed
\end{proof}

\begin{remark}
From Lemma~\ref{lem:multiplicity} we conclude the multiplicity being one is a necessary but not sufficient condition for toric initial ideals. A cone can have multiplicity one but its associated initial ideal might be neither prime nor binomial. There may be associated primes that contain monomials in the decomposition of $\init_{\textbf w}(I)$ and these do not contribute to the multiplicity. We list examples of such cones in $\trop(\Flag_5)$ as we will see in Theorem~\ref{flag5}.
\end{remark}

Let $I$ be a  homogeneous ideal in $S$ such that the Krull dimension of $S/I$ is $d$. Consider $\trop(V(I))\subset\mathbb R^{n+1}/H$ and the $d$-dimensional subfan  $F\subset \mathbb R^{n+1}$ of the Gr\"obner fan of $I$ with $F/H\cong \trop(V(I))$. 
When $V(I)\subset \mathbb P^{k_1}\times\cdots\times \mathbb P^{k_s}$ the linear space $H$ is spanned by the rows of the matrix $D$. In particular, when $V(I)\subset \mathbb P^n$ we have that $H$ is equal to the span of $(1,\ldots,1)$.
We now describe some properties of the toric initial ideals corresponding to  maximal cones of $\trop(V(I))$.
Let $C$ be a cone in $\trop(V(I))$ and $\{{\bf w}_1,\ldots,{\bf w}_d\}$ be $d$ linearly independent vectors in $F$ generating the maximal cone $C'$, such that $C'/H\cong C$.
We can assume that the ${\bf w}_i$'s have integer entries since $F$ is a rational fan. We define the  matrix associated to $C$ to be 
\begin{equation}\label{def:W}
W_C=[{\bf w}_1,\ldots,{\bf w}_d]^T.
\end{equation}

Consider a sublattice $L$ of $\mathbb Z^{n+1}$ and the standard basis $e_1,\dots, e_{n+1}$ of $\mathbb Z^{n+1}$. Given $\ell=(\ell_1,\dots,\ell_{n+1})\in L$ we set $\ell^+=\sum_{\ell_i>0}\ell_i e_i$ and $\ell^-=-\sum_{\ell_j<0}\ell_j e_j$. Note that $\ell=\ell^+-\ell^-$ and $\ell^+,\ell^-\in \mathbb N^{n+1}$. We use here the same notation of \cite[page 15]{Cox2}.

\begin{lemma}\label{toric:lemma}
Let $I$ be a homogeneous ideal in $S$ and $C$ a maximal cone in $ \trop(V(I))$. If $\init_{C}(I)$ is toric, then there exists a sublattice $L$ of $\mathbb Z^{n+1}$ and constants $0\neq c_\ell\in\Bbbk$ with $\ell \in L$ such that 
\[
\init_{C}(I)=I(W_C):=\langle \textbf{x}^{\ell^+}-c_{\ell}\textbf {x}^{\ell^-}:\ \ell\in L \rangle. 
\]
In particular, $L$  is  the kernel of the map $f:\mathbb Z^{n+1}\to \mathbb Z^{d}$ defined by the matrix $W_C$. 
If  $C$ has multiplicity one and  $\init_{C}(I)$ is not toric, then  the unique toric ideal in the primary decomposition of  $\init_{C}(I)$ is of the form $ I(W_C)$.
\end{lemma}

\begin{proof}
Let  $\init_{C}(I)\subset  S$ be a toric initial ideal and let $C'$ be the corresponding cone in $F$.
The fan structure is defined on  $\trop(V(I))$ so that for every $\textbf{w}',\textbf{w}$ in the relative interior of $ C'$ we have $\init_{\textbf{w}'}(I)=\init_{C}(I)=\init_{\textbf{w}}(I)$. This implies $\init_{C}(I)$ is $W_C$-homogeneous with respect to the $\mathbb {Z}^d$-grading on $S$ given by the matrix $W_C$. By \cite[Lemma~10.12]{Stu96} there exists an automorphism $\phi$ of $S$ sending $x_i$ to $\lambda_i x_i$ for some $\lambda_i\in \Bbbk$, such that the ideal $\init_{C}(I)$ is isomorphic to an ideal of the form 
\[
I_{L}:=\langle \textbf{x}^{\ell^+}-\textbf{x}^{\ell^-}:\ \ell\in L \rangle. 
\]
Here $L$ is the sublattice of $\mathbb Z^{n+1}$ given by the  kernel of the map $f:\mathbb Z^{n+1}\to \mathbb Z^{d}$. 
Applying $\phi^{-1}$ to $\init_{C}(I)$  we can write each toric initial ideal as
\[
\langle \textbf{x}^{\ell^+}-c_{\ell}\textbf {x}^{\ell^-}:\ \ell\in L \rangle=I(W_C), 
\]
for some $ 0\neq c_\ell\in\Bbbk$, $L$ and $W_C$ defined above. 

\smallskip
Let $C$ be a cone of multiplicity one and suppose $\init_{C}(I)$ is not prime. Then by Lemma~\ref{lem:multiplicity} there exists a unique toric ideal $J$ in the primary decomposition of $\init_{C}(I)$. This toric ideal $J$  contains $\init_{C}(I)$ and we will show that it can be expressed as $I(W_C)$.  The variety $V(I)$ is considered as subvariety of $\mathbb P^n$. As in Lemma~\ref{lem:multiplicity}, the case $V(I)\subset \mathbb P^{k_1}\times\cdots\times \mathbb P^{k_s}$ has an analogous proof.

The tropical variety depends only on the intersection of $V(I)$ with the torus,  and $\init_{C}(I)\Bbbk[x_1^{\pm 1},\ldots ,x_n^{\pm 1}]$  is equal to $J$. Hence, $J$ is a prime ideal that is homogeneous with respect to $W_C$ so we can proceed as above to show $J$ can be written as $\langle \textbf{x}^{\ell^+}-c_{\ell}\textbf {x}^{\ell^-}:\ \ell\in L \rangle=I(W_C)$.  
\qed
\end{proof}

\begin{remark}
Note that the lattice $L$ and the ideal $I(W_C)$ only depend on the linear space spanned by the rays of the cone $C'$. Hence they are the same for every set of $d$ independent vectors in $C'$ chosen to define $W_C$.
\end{remark}

\section{Tropicalization and toric degenerations }\label{sec:3}

In this section we study the tropicalization of $\Flag_4 $ and $\Flag_5$. We analyze the  Gr\"obner toric degenerations arising from $\trop(\Flag_4)$ and $\trop(\Flag_5)$, and we compute the polytopes associated to their normalizations. In Proposition~\ref{6config} we describe the \textit{tropical configurations} arising from the maximal cones of $\trop(\Flag_4)$. These are configurations of a point on a tropical line in a tropical plane corresponding to the points in the relative interior of a maximal cone.

\medskip
Before stating our main results, we recall the following definition. 

\begin{definition}
There exists a \emph{unimodular equivalence} between two lattice polytopes $P$ and $Q$ (resp. two fans $\mathcal F$ and $\mathcal G$)
 if there exists an affine lattice isomorphism  $\phi$ of the ambient lattices sending the vertices (resp. the rays) of one polytope (resp. fan) to the vertices (resp. rays) of the other. Moreover, if $\sigma$ is a face of $P$ (resp. of $\mathcal F$) then $\phi(\sigma)$ is a face of $Q$ (resp. $\mathcal{G}$) and the adjacency of faces is respected.  
\end{definition}

\begin{remark}
We are interested in finding distinct fans up to unimodular equivalence as they give rise to non-isomorphic toric varieties. Often it will  be possible only to determine combinatorial equivalence (see \cite[\S 2.2]{Cox2}). This  is a weaker condition but when it does not hold it 
allows us to rule out unimodular equivalence.
\end{remark}

\begin{theorem}\label{flag4}
The tropical variety $\trop(\Flag_4)$ is a $6$-dimensional rational fan in $\mathbb{R}^{14}/\mathbb R^3$ with a $3$-dimensional lineality space. It consists of 78 maximal cones, 72 of which are prime. They are organized in five $S_4\rtimes \mathbb Z_2$-orbits, four of which contain  prime cones. The prime cones give rise to four 
non-isomorphic toric degenerations.
\end{theorem}
\begin{proof}
The theorem is proved by explicit computations. We developed a   $\emph{Macaulay2}$ package called $\mathtt{ToricDegenerations}$ containing all the functions we use. The package and the  data needed for this proof are available at 
\[\mathtt{https://github.com/ToricDegenerations}.\]
The flag variety $\Flag_4$ is a $6$-dimensional subvariety of
$\mathbb{P}^3\times \mathbb{P}^5 \times \mathbb{P}^3$.
The ideal $I_4$ defined  in the previous section is contained in the total coordinate ring  $R$ of $\mathbb{P}^3\times \mathbb{P}^5 \times \mathbb{P}^3$ which is 
the polynomial ring  over $\mathbb C$ on the variables
\[
p_1,p_2,p_3,p_4,p_{12},p_{13},p_{14},p_{23},p_{24},p_{34},p_{123},p_{124},p_{134},p_{234}.
\]
The grading on  $R$ is given by the matrix \begin{equation}\label{degree}
\setcounter{MaxMatrixCols}{14}
D=\begin{bmatrix}
 1 & 1 & 1 & 1 & 0 & 0 & 0 & 0 & 0 & 0 & 0 & 0 & 0 & 0 \\
 0 & 0& 0 &0 &1 &1&1&1&1&1&0&0&0&0\\
 0&0&0&0&0&0&0&0&0&0&1&1&1&1
\end{bmatrix}.\end{equation}

The explicit form of $I_4$ can be found in \cite[page 276]{MS05}. 
As we have seen in \S\ref{sec:background} the tropicalization of $\Flag_4$ is contained in $\mathbb R^{14}/H$.  In this case $H$ is the vector space spanned by the rows of $D$.

We use the $\emph{Macaulay2}$ \cite{M2} interface to $\emph{Gfan}$ \cite{Gfan} to compute $\trop(\Flag_4)$. The given input is the ideal $I_4$ and the $S_4\rtimes \mathbb Z_2$-action (see \cite[\S3.1.1]{GfanM}). The output is a subfan $F$ of the Gr\"obner fan of dimension  $9$. We quotient it by $H$ to get $\trop(\Flag_4)$ as a $6$-dimensional fan contained in $\mathbb R^{14}/H\cong \mathbb R^{14}/\mathbb {R}^3$.

Firstly, the function \texttt{computeWeightVectors} computes a list of vectors. There is  one for every maximal cone of $\trop(\Flag_4)$ and it is contained in the relative interior of the corresponding cone. Then  \texttt{groebnerToricDegenerations} computes all the initial ideals and  checks if they are binomial and prime over $\mathbb Q$. These are organized in a hash table, which is the output of the function.
All 78 initial ideals are binomial and all maximal cones have multiplicity one. In order to check primeness over $\mathbb C$, we have to check if $\init_C(I_4)=I(W_C)$. This can be done by computing  the  degrees of  $V(\init_{C}(I_4))$ and $V(I(W_C))$ seen as subvarieties of $\mathbb P^{13}$. If these are equal, then there are no  non-toric ideals in the primary decomposition of $\init_{C}(I_4)$. Note that the degree of $V(I(W_C))$ equals the degree of $V(I_L)$, where $L$ and $I_L$ can be computed from $W_C$ as in the proof of Lemma~\ref{toric:lemma}. 

We consider the orbits of the $S_4\ltimes \mathbb Z_2$-action on the set of initial ideals. These correspond to the orbits of maximal cones of $F$ and hence of $\trop(\Flag_4)$. There is one orbit of non-prime initial ideals and four orbits of prime initial ideals. 
The varieties corresponding to initial ideals contained in the same orbit are isomorphic. Therefore, for each orbit we choose a representative of the form $\init_{C}(I_4)=I(W_C)$ for some cone $C$. 

We now compute for each of the four prime orbits, the polytope of the normalization of the associated toric varieties. We use the \emph{Macaulay2}-package \emph{Polyhedra}~\cite{Poly2} for the following computations. 
 
The lattice $M$ associated to $S/I(W_C)$ is generated over $\ZZ$ by the columns of $W_C$. 
To use \emph{Polyhedra} we want to have a lattice with index $1$ in $\mathbb Z^{9}$.
Hence, in case the index of $M$ in $\mathbb Z^9$ is different from $1$, we consider $M$ as the lattice generated by the columns of the matrix  $(\ker( (\ker( W_C))^T)^{T}$.
Here, for a matrix $A$ we consider $\ker(A)$ to be the matrix whose columns minimally generate the kernel of the map $\mathbb Z^{14}\to\mathbb Z^9$ defined by $A$.
We denote the set of generators  of $M$ by $\mathcal B_C=\{{\bf b}_1,\ldots ,{\bf b}_{14}\}$ so that 
$M=\mathbb Z\mathcal B_C$.

The toric variety $\mathbb P^3\times \mathbb P^5 \times \mathbb P^3$ can be seen as $\Proj(\oplus_{\ell} R_{\ell(1,1,1)})$ and $I(W_C)$
as an ideal in $\oplus_{\ell} R_{\ell(1,1,1)}$ (see \cite[Chapter 10]{MS05}). The associated toric variety is $\Proj(\oplus_{\ell} \mathbb C[\mathbb N \mathcal B_C]_{\ell(1,1,1)})$.
The polytope $P$ of the normalization is given as the convex hull of those lattice points in $\mathbb N \mathcal B_C$ corresponding to degree $(1,1,1)$-monomials in $\mathbb C[\mathbb N \mathcal B_C]$.

These can be found in the following way. We order the rows of the matrix $({\bf b}_1,\ldots,{\bf b}_{14})$ associated to $\mathcal B_C$ so that the first three rows give the matrix $D$ from \eqref{degree}. Now the matrix $({\bf b}_1,\ldots,{\bf b}_{14})$ represents a map $\ZZ^{14}\to\ZZ^3\oplus\ZZ^6$, where $\ZZ^3\oplus\ZZ^6$ is the lattice $M$ and the $\ZZ^3$ part gives the degree of the monomials associated to each lattice point on $M$.
The lattice points, whose convex hull give the polytope $P$, are those ones with the first three coordinates being $1$. 
In other words, we have obtained $P$ by applying the reverse procedure of constructing a toric variety from a polytope (see \cite[\S 2.1-\S 2.2]{Cox2}). 
Note that the difference from the procedure given in \cite[\S 2.1-\S 2.2]{Cox2} is the $\mathbb Z^3$-grading and because of that we do not consider  the convex hull of $\mathcal B_C$, but the intersection of $\mathbb N \mathcal B_C$ with these hyperplanes.

In Table~\ref{fig:num} there are the numerical invariants of the initial ideals and their corresponding  polytopes.
Using \emph{polymake} \cite{GJ00} we first obtain that there is no combinatorial equivalence between each pair of polytopes. 
This means that there is no unimodular equivalence between the corresponding normal fans, 
hence the normalization of the toric varieties associated to these toric degenerations are not isomorphic. 
This implies that we obtain four non-isomorphic  toric degenerations.  \qed
\end{proof}

\begin{table}
\begin{center}
\begin{tabular}{|l|l|l|l|l|l|l|}
 \hline
 Orbit&Size&Cohen-Macaulay&Prime&$\#$Generators&F-vector of associated polytope \\
 \hline
1&  24 &  Yes& Yes &10&(42, 141, 202, 153, 63, 13)\\
2&  12 &  Yes& Yes &10&(40, 132, 186, 139, 57, 12)\\
3&  12 &  Yes & Yes &10&(42, 141, 202, 153, 63, 13)\\
4&    24 & Yes & Yes &10&(43, 146, 212, 163, 68, 14)\\
5&    6 & Yes & No &10& Not applicable\\
\hline
\end{tabular}
\end{center}
\caption{The tropical variety $\trop(\Flag_4)$ has 78 maximal cones organized in five $S_4\rtimes \mathbb Z_2$-orbits. The algebraic invariants of the initial ideals associated to these cones and the F-vectors of their associated polytopes are listed here.}\label{fig:num} 
\end{table}

\begin{proposition}\label{6config}
There are six tropical configurations up to symmetry (depicted in Figure~\ref{figure:1}) arising from the maximal cones of $\trop(\Flag_4)$. They are further organized in five $S_4\rtimes \mathbb Z_2$-orbits.
\end{proposition}

\begin{proof}
The tropical variety $\trop(\Flag_4)$ is contained in 
\[
\trop(\Gr(1,\mathbb C^4))\times\trop(\Gr(2,\mathbb C^4))\times\trop(\Gr(3,\mathbb C^4)).
\]
Each tropical Grassmannian parametrizes tropicalized linear spaces (see \cite[Theorem 4.3.17]{M-S}). This implies that  every point $p$ in $\trop(\Flag_4)$ corresponds to a  chain of tropical linear subspaces given by a point on a tropical line contained in a tropical plane. All  tropical chains are \textit{realizable}, meaning that they are the tropicalization of the classical chains of linear spaces of $\Bbbk^4$ corresponding to a point $q$ in $\Flag_4$  such that $\val(q)=p$, where $\Bbbk=\mathbb C\!\{\!\{t\!\}\!\}$ and $\val$ is the natural valuation on this field (see \cite[Part (3) of Theorem 3.2.3]{M-S}).

In this case, there is only one combinatorial type for the tropical plane and four possible types  for the lines up to symmetry (see \cite[Example 4.4.9]{M-S}). The plane consists of six $2$-dimensional cones positively spanned by all possible pairs of vectors  $(1,0,0)^{T},(0,1,0)^{T},(0,0,1)^{T}$, and $(-1,-1,-1)^{T}$. The combinatorial types of the tropical lines are shown in Figure \ref{Comb types}. The leaves of these graphs represent the rays of the tropical line labeled $1$ up to $4$ corresponding to the positive hull of each of the vectors 
$(1,0,0)^{T},(0,1,0)^{T},(0,0,1)^{T}$, and $(-1,-1,-1)^{T}$.

\begin{figure}
    \begin{center}
    \begin{tikzpicture}[scale=.4]
\draw (-2,0) -- (-1,1) -- (1,1) -- (2,0);
\draw (-2,2) -- (-1,1);
\draw (1,1) -- (2,2);
\node at (-2.2,0) {2};
\node at (-2.2,2) {1};
\node at (2.2,0) {4};
\node at (2.2,2) {3};

\draw (3.5,0) -- (4.5,1) -- (6.5,1) -- (7.5,0);
\draw (3.5,2) -- (4.5,1);
\draw (6.5,1) -- (7.5,2);
\node at (3.3,0) {3};
\node at (3.3,2) {1};
\node at (7.7,0) {4};
\node at (7.7,2) {2};

\draw (9,0) -- (10,1) -- (12,1) -- (13,0);
\draw (9,2) -- (10,1);
\draw (12,1) -- (13,2);
\node at (8.7,0) {4};
\node at (8.7,2) {1};
\node at (13.2,0) {3};
\node at (13.2,2) {2};

\draw (15.3,0) -- (16.3,1);
\draw (15.3,2) -- (16.3,1);
\draw (16.3,1) -- (17.3,2);
\draw (16.3,1) -- (17.3,0);
\node at (15,0) {4};
\node at (15,2) {1};
\node at (17.6,0) {3};
\node at (17.6,2) {2};

\end{tikzpicture}
\caption{Combinatorial types of tropical lines in $\mathbb R^4/ \mathbb R\bf{1}$.}\label{Comb types}
\end{center}
\end{figure}
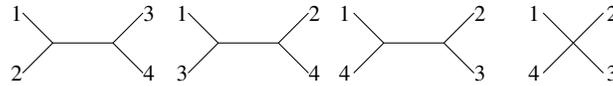

Consider the $S_4\rtimes \mathbb Z_2$-orbits of maximal cones of $\trop(\Flag_4)$. If we compute the chain of tropical linear spaces corresponding to an element in each orbit, we get the configurations in  Figure~\ref{figure:1}. Note that we do not include the labeling since up to symmetry we can get all possibilities.
The point on the line is the black dot. In case the intersection of the line with the rays of the plane is the vertex of the plane then we denote this with a  hollow dot. A vertex of the line is colored in gray if it   lies on a ray of the plane.
For example in orbit 2,  label the rays $1$ to $4$ anti-clockwise starting from the top left edge. We have rays $1$ and $2$ 
in the $2$-dimensional positive hull of $(1,0,0)^{T}$ and $(0,1,0)^{T}$. The vector associated to the internal edge is $(1,1,0)^{T}$. 
The gray point is the origin and the black point has coordinates $(a,1,0)^{T}$ for $a>1$.

Orbits 1 and 4 in Figure~\ref{figure:1} have size $24$, orbits 2 and 3 have  size $12$ and orbit 5 has size $6$. Note that orbit 5 corresponds to non-prime initial ideals. Orbit 1 contains two combinatorial types of tropical configurations and one is sent to the other by the $\mathbb Z_2$-action on the tropical variety.
The orbits $2$ and $3$ differ from the fact that for each combinatorial type of line the gray dot can lie on one of the four rays of the tropical plane. These possibilities are grouped in two pairs, one is in orbit $2$ and the other in orbit $3$. 
\qed
\end{proof}

\begin{figure}
    \begin{center}
    \begin{tikzpicture}[scale=.7]
\begin{scope}[xshift=-.8cm]
\node at (-1.5,1) { Orbit 1};
\draw (0,0) -- (1,1) -- (3,1) -- (4,0);
\draw (0,2) -- (1,1);
\draw (3,1) -- (4,2);

\draw [fill] (3.5,1.5) circle [radius=0.15];
\draw [fill, white] (1.5,1) circle [radius=0.15];
    \draw [gray, ultra thick] (1.5,1) circle [radius=0.
    15];
   
\begin{scope}[xshift=5.5cm]
\draw (0,0) -- (1,1) -- (3,1) -- (4,0);
\draw (0,2) -- (1,1);
\draw (3,1) -- (4,2);

\draw [fill, gray] (3,1) circle [radius=0.15];
\draw [fill] (1.5,1) circle [radius=0.15];    
\end{scope}   
\end{scope}
\begin{scope}[yshift=-3cm, xshift=-.75cm]
\node at (-1.5,1) {Orbit 2};
\draw (0,0) -- (1,1) -- (3,1) -- (4,0);
\draw (0,2) -- (1,1);
\draw (3,1) -- (4,2);

\draw [fill] (.5,1.5) circle [radius=0.15];
\draw [fill, gray] (3,1) circle [radius=0.15];

\begin{scope}[xshift=8cm]
\node at (-1.5,1) { Orbit 3};
\draw (0,0) -- (1,1) -- (3,1) -- (4,0);
\draw (0,2) -- (1,1);
\draw (3,1) -- (4,2);

\draw [fill] (3.5,1.5) circle [radius=0.15];
\draw [fill, gray] (1,1) circle [radius=0.15];
\end{scope} 

\begin{scope}[yshift=-3cm]
\node at (-1.5,1) {Orbit 4};
\draw (0,0) -- (1,1) -- (3,1) -- (4,0);
\draw (0,2) -- (1,1);
\draw (3,1) -- (4,2);

\draw [fill] (3.5,1.5) circle [radius=0.15];
\draw [fill, gray] (3,1) circle [radius=0.15];

\begin{scope}[xshift=8cm]
\node at (-1.5,1) {Orbit 5};
\node at (-1.5,0.5) {(Non-prime)};
\draw (0,0) -- (1,1) -- (3,1) -- (4,0);
\draw (0,2) -- (1,1);
\draw (3,1) -- (4,2);

\draw [fill] (2.5,1) circle [radius=0.15];
\draw [fill, white] (1.5,1) circle [radius=0.15];
    \draw [gray, ultra thick] (1.5,1) circle [radius=0.15];    
\end{scope}
\end{scope}   
\end{scope}
\begin{scope}[yshift=-8.5cm]
\draw [fill] (3,0) circle [radius=0.15];
\draw [fill, gray] (3,.5) circle [radius=0.15];
\draw [fill, white] (3,1) circle [radius=0.15];
    \draw [gray, ultra thick] (3,1) circle [radius=0.15]; 
\node at (4.1,0) { \ the point};  
\node at (5.75,.5) { a point on a ray of the plane};
\node at (5.3,1) {the vertex of the plane};
\end{scope}
\end{tikzpicture}
    \end{center}
    \caption{The list of all tropical configurations up to symmetry that arise in $\Flag_4$. The hollow and the full gray dot denote whether that vertex of the line is the vertex of the plane  or it is contained in a ray of the plane. The black dot is the position of the point on the line.}
   \label{figure:1} 
\end{figure}

\begin{theorem}\label{flag5}
The tropical variety $\trop(\Flag_5)$ is a $10$-dimensional fan in $\mathbb{R}^{30}/\mathbb {R}^4$ with a $4$-dimensional lineality space. It consists of $69780$ maximal cones which are grouped in $536$ $S_5\rtimes \mathbb Z_2$-orbits.
These give rise to $531$ orbits of binomial initial ideals and among these $180$ are prime. They correspond to $180$ non-isomorphic toric degenerations.  
\end{theorem}

\begin{proof}
The flag variety $\Flag_5$ is a $10$-dimensional variety defined by 66
quadratic polynomials in the total coordinate ring of $\mathbb P^4\times \mathbb{P}^{9}\times \mathbb{P}^{9}\times \mathbb P^4$. These are of the form $\sum_{j\in J\backslash I}(-1)^{l_j}p_{I\cup\{j\}}p_{J\backslash\{j\}}$, where $J,I\subset \{1,\ldots,5\}$ and $l_j=\#\{k\in J: j<k\}+\#\{i\in I:i<j\}$. 

The proof is similar to the proof of Theorem~\ref{flag4}.
The only difference is that  the action of $S_5\rtimes \mathbb Z_2$  on $\Flag_5$ is crucial for the computations. In fact, without exploiting the symmetries the calculations to get the tropicalization would not terminate. Moreover, we only verify primeness  of the initial ideals over $\mathbb Q$ using the \emph{primdec} library \cite{Primedec} in \emph{Singular}
\cite{DGPS}. We compute the polytopes associated to the normalization of the $180$  toric varieties in the same way as Theorem~\ref{flag4}, only changing the matrix of the grading. It is now given by
\begin{equation}\label{degree2}
\setcounter{MaxMatrixCols}{30}
\begin{bmatrix}
 1 & 1 & 1 & 1 &1& 0 & 0 & 0 & 0 & 0 & 0 & 0 & 0 & 0 & 0 &0 & 0 & 0 & 0 & 0 & 0 & 0 & 0 & 0 & 0 &0&0&0&0&0\\
 0 & 0& 0 &0 &0&1 &1&1&1&1&1&1&1&1&1&0 & 0 & 0 & 0 & 0 & 0 & 0 & 0 & 0 & 0&0&0&0&0&0\\
 0 & 0& 0 &0 &0 &0 & 0 & 0 & 0 & 0 & 0 & 0 & 0 & 0 & 0 &1 &1&1&1&1&1&1&1&1&1&0&0&0&0&0\\
  0&0 & 0& 0 &0 &0 &0 & 0 & 0 & 0 & 0 & 0 & 0 & 0 & 0 & 0 & 0 & 0 & 0 & 0 & 0 & 0 & 0 & 0 & 0 &1 & 1 & 1 & 1 &1
\end{bmatrix}.
\end{equation}

Since there are no combinatorial equivalences among the normal fans to these polytopes, we deduce that the obtained toric degenerations are pairwise non-isomorphic. More information on the non-prime initial ideals is available in Table~\ref{nonprime} in the appendix.
\qed
\end{proof}

\section{String polytopes and the FFLV-polytope}\label{sec:rep}
This section provides an introduction to string cones, string polytopes, and the FFLV polytope with explicit computations for $\Flag_4$ and $\Flag_5$. String polytopes are described by Littelmann in \cite{Lit98}, and by Berenstein and Zelevinsky in \cite{BZ01}. FFLV stands for Feigin, Fourier, and Littelmann, 
who defined this polytope in \cite{FFL11}, and Vinberg who conjectured its existence in a special case. Both, the string polytopes and the FFLV polytope, can be used to obtain toric degenerations of the flag variety.

\medskip

Let $W=S_n$ be the symmetric group, which is the Weyl group corresponding to $G=SL_n$ over $\mathbb C$ with the longest word $w_0$ given in the alphabet of simple transpositions $s_i=(i,i+1)\in S_n$.
We choose the Borel subgroup $B\subset SL_n$ of upper triangular matrices and the maximal torus $T\subset B$ of diagonal matrices. Further, let $U^-\subset B^-$ be the unipotent radical in the opposite Borel subgroup, i.e. the set of lower triangular matrices with $1$'s on the diagonal.
Let $\text{Lie}(G)= \mathfrak g= \mathfrak{sl}_n$ be the corresponding Lie algebra, i.e. $n\times n$-matrices with trace zero. Let $\mathfrak h=\text{Lie}(T)\subset \mathfrak g$ be diagonal matrices. We fix a Cartan decomposition $\mathfrak g= \mathfrak n^- \oplus \mathfrak b$ with $\text{Lie}(B)=\mathfrak b$ and $\text{Lie}(U^-)=\mathfrak n^-$. Note that $SL_n/B=\Flag_n$.
By $R$ we denote the root system  of $\mathfrak g$ (see \cite[Section 9.2]{H78} for the definition). Here $R$ is of type $\mathtt{A}_{n-1}$. Let $R^+$ be the set of positive roots with respect to the given choice of $\mathfrak b$.
We denote the simple roots generating the root lattice by $\alpha_1,\ldots,\alpha_{n-1}$, and their coroots generating the dual lattice by $\alpha_i^\vee$. For positive roots $\alpha_i+\alpha_{i+1}+\dots +\alpha_j$ with $j\ge i$ we use the short notation $\alpha_{i,j}$. Note that using this notation we have $\alpha_{i,i}=\alpha_i$. The number of positive roots is $N$, which is also the length of $w_0$ as reduced expression in the $s_i$. 
For a positive root $\beta \in R^+$, $f_\beta$ is a non-zero root vector in $\mathfrak n^-$ of weight $-\beta$.
Let $P$ denote the weight lattice of $T$ generated by the fundamental weights $\omega_1,\ldots,\omega_{n-1}$. The definition can be found in \cite[Section 13.1]{H78}.
A weight $\lambda\in P$ is \textit{regular dominant}, if $\lambda=\sum_{i=1}^{n-1}a_i\omega_i$ with $a_i\in \mathbb Z_{> 0}$ for all $i$.
The subset of regular dominant weights is denoted by $P^{++}$. 

\begin{figure}
\begin{center}
\begin{tikzpicture}[scale=.7]

\node at (0,2) {$l_3$};
\node at (0,1) {$l_2$};
\node at (0,0) {$l_1$};

\node at (7,0) {$L_3$};
    \draw [fill] (6.5,0) circle [radius=0.05];
\node at (7,1) {$L_2$};
    \draw [fill] (6.5,1) circle [radius=0.05];
\node at (7,2) {$L_1$};
    \draw [fill] (6.5,2) circle [radius=0.05];

\node[below] at (2.5,-.5) {$s_{1}$};
\node[below] at (3.5,-.5) {$s_{2}$};
\node[below] at (4.5,-.5) {$s_{1}$};

\node[above] at (3.5,1.65) {$v_{1,3}$};
    \draw [fill] (3.5,1.5) circle [radius=0.05];
\node[right] at (4.6,0.5) {$v_{2,3}$};
    \draw [fill] (4.5,0.5) circle [radius=0.05];
\node[left] at (2.4,0.5) {$v_{1,2}$};
    \draw [fill] (2.5,0.5) circle [radius=0.05];

\draw[rounded corners] (0.5,0) --(1.25,0) -- (2,0)-- (3,1) --(4,2) -- (5.25,2);
    \draw[->, rounded corners] (1.5,0) -- (1.25,0);
    \draw[->, rounded corners] (3.5,1.5) -- (3,1);
    \draw[->, rounded corners] (6.5,2) -- (5.25,2);
\draw[rounded corners] (1.25,1) -- (2,1) -- (3,0) -- (3.5,0)-- (4,0) -- (5,1) -- (5.57,1)-- (6.5,1);
    \draw[->] (0.5,1) -- (1.25,1);
    \draw[->] (3.25,0) -- (3.5,0);
    \draw[->] (5.5,1) -- (5.75,1);
\draw[rounded corners] (1.5,2) -- (3,2) -- (4,1)-- (5,0) -- (5.75,0) -- (6.5,0);
    \draw[->] (0.5,2) -- (1.5,2);
    \draw[->] (3.5,1.5) -- (4,1);
    \draw[->] (5.5,0) -- (5.75,0);

\end{tikzpicture}
\end{center}
\caption{Pseudoline arrangement corresponding to $\underline w_0=s_1s_2s_1$ for $\Flag_3$ with orientation induced by $l_1$.}\label{fig:pseudo} 
\end{figure}
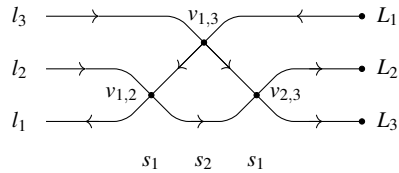
  
\medskip  
For a fixed weight $\lambda\in P^{++}$ and a reduced expression $\underline w_0$ of $w_0$ we construct the string polytope $Q_{\underline w_0}(\lambda)$.
This description can be found in \cite{GP00} and \cite{Lit98}. To $\underline w_0$ one associates a \emph{pseudoline arrangement}. 
It consists of $n$ horizontal \emph{pseudolines} 
(or in short \emph{lines}) labeled $1$ to $n$ on the left from bottom to top. Pairwise, they cross exactly once and the order of crossings depends on $\underline w_0$.
More precisely, a simple reflection $s_i$ induces a crossing on level $i$, see Figure~\ref{fig:pseudo}. 
The diagram has vertices $v_{i,j}$ for every crossing of lines $l_i$ and $l_j$, as well as vertices $L_1,\dots, L_n$ from top to bottom at the right ends of the lines.
Every line $l_i$ with $1\le i < n$ induces an orientation of the diagram obtained by orienting $l_j$ for $j>i$ from left to right and $l_k$ for $k\le i$ from right to left.

Fix an oriented path $v_0\to \dots \to v_s$ in an (oriented) pseudoline arrangement and assume
three adjacent vertices $v_{k-1}\to v_{k}\to v_{k+1}$ on the path belong to the same pseudoline $l_i$. Whenever a path does not change the line at a crossing, we are in this situation.
Let $v_k$ be the intersection of $l_i$ and $l_j$.
The path is \emph{rigorous}, if it avoids the following two situations:
\begin{itemize}
\item $i<j$ and both lines are oriented to the left or
\item $i>j$ and both lines are oriented to the right.
\end{itemize}
The first situation is visualized on the left of Figure~\ref{fig:rigorous} and the second on the right. The thick arrow is the part of line $l_i$ that must not be contained in a rigorous path. We denote by $\mathcal P_{\underline w_0}$ the set of all possible rigorous paths for all orientations induced by the lines $l_i$ with $1\le i<n$.

\begin{figure}[h]
\begin{center}
\begin{tikzpicture}

\draw[rounded corners] (3,0) -- (2,0) -- (1,1) -- (0,1);
        \draw[->] (3,0) -- (2.5,0);
        \draw[->] (0.8,1) -- (.5,1);
\draw[rounded corners] (3,1) -- (2,1) -- (1,0) -- (0,0);
    \draw[->] (3,1) -- (2.5,1);
    \draw[->] (0.8,0) -- (.5,0);
\draw[->, ultra thick] (1.9,0.9) -- (1.1,0.1);

\begin{scope}[xshift=5cm]
  \draw[rounded corners] (3,0) -- (2,0) -- (1,1) -- (0,1);
        \draw[->] (.5,1) -- (.8,1);
        \draw[->] (2.2,0) -- (2.5,0);
\draw[rounded corners] (3,1) -- (2,1) -- (1,0) -- (0,0);
    \draw[->] (.5,0) -- (.8,0);
    \draw[->] (2.2,1)-- (2.5,1);
\draw[<-, ultra thick] (1.9,0.1) -- (1.1,0.9);
\end{scope}
\end{tikzpicture}
\end{center}
\caption{The two local orientations with thick arrows forbidden in rigorous paths.}\label{fig:rigorous}
\end{figure}
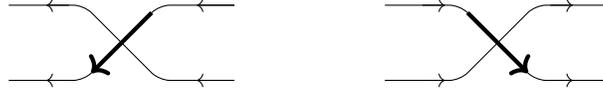

\begin{example}
Consider $\Flag_4$ with reduced expression $\underline w_0=s_1s_2s_3s_2s_1s_2$. We draw the corresponding pseudoline arrangement in Figure~\ref{rig.path} with orientation induced by $l_1$. The rigorous paths for this orientation have source $L_1$ and sink $L_2$. An example of a \emph{rigorous} path is
\[
{\bf p}= L_1\to v_{1,4} \to v_{1,3}\to v_{3,4} \to v_{2,3}\to L_2.
\]
An example for a \emph{non-rigorous} path is one that passes through a thick arrow, for example
\[
{\bf p}'=L_1\to v_{1,4}\to v_{3,4} \to v_{2,4}\to v_{2,3}\to L_2.
\]
\end{example}

\begin{figure}
\begin{center}
\begin{tikzpicture}
\draw[rounded corners] (0,0) -- (1,0) -- (4,3) -- (8,3);
    \draw[->] (.75,0) -- (.5,0);
    \draw[->] (2.5,1.5) -- (2,1);
    \draw[->] (3.5,2.5) -- (3,2);
    \draw[->] (7,3) -- (6.5,3);
\draw[rounded corners] (0,1) -- (1,1) -- (2,0) -- (5,0) -- (7,2) -- (8,2);
    \draw[->] (0.25,1) -- (.5,1);
    \draw[->] (3,0) -- (3.5,0);
    \draw[->] (5.5,0.5) -- (6,1);
    \draw[->] (7.1,2) -- (7.5,2);
\draw[rounded corners] (0,2) -- (2,2) -- (3,1) -- (4,1) -- (5,2) -- (6,2) -- (7,1) -- (8,1);
    \draw[->] (1,2) -- (1.5,2);
    \draw[->] (3.1,1) -- (3.5,1);
    \draw[->] (5.1,2) -- (5.5,2);
    \draw[->] (7.1,1) -- (7.5,1);
\draw[rounded corners] (0,3) -- (3,3) -- (6,0)  -- (8,0);
    \draw[->] (1,3) -- (1.5,3);
    \draw[->] (3.5,2.5) -- (4,2);
    \draw[->] (4.5,1.5) -- (5.1,0.9);
    \draw[->] (6.5,0) -- (7,0);

\draw[->, ultra thick] (4.2,1.8) -- (4.8,1.2);
\draw[->, ultra thick] (5.2,0.8) -- (5.8,0.2);
\draw[->, ultra thick] (6.2,1.8) -- (6.8,1.2);

\node at (1.5,-.5) {$s_1$};
\node at (2.5,-.5) {$s_2$};
\node at (3.5,-.5) {$s_3$};
\node at (4.5,-.5) {$s_2$};
\node at (5.5,-.5) {$s_1$};
\node at (6.5,-.5) {$s_2$};

\draw [fill] (8,2) circle [radius=0.03];
\draw [fill] (8,3) circle [radius=0.03];
\draw [fill] (3.5,2.5) circle [radius=0.03];
\draw [fill] (2.5,1.5) circle [radius=0.03];
\draw [fill] (4.5,1.5) circle [radius=0.03];
\draw [fill] (6.5,1.5) circle [radius=0.03];
\draw [fill] (5.5,.5) circle [radius=0.03];

\node at (8.5,2) {$L_2$};
\node at (8.5,3) {$L_1$};
\node at (3,2.5) {$v_{1,4}$};
\node at (2,1.5) {$v_{1,3}$};
\node at (4,1.5) {$v_{3,4}$};
\node at (6,1.5) {$v_{2,3}$};
\node at (5,0.5) {$v_{2,4}$};

\end{tikzpicture}
\caption{A pseudoline arrangement for $\Flag_4$ with $\underline w_0=s_1s_2s_3s_2s_1s_2$ and orientation induced by $l_1$. Thick arrows denote forbidden line segments for rigorous paths.}\label{rig.path}
\end{center}
\end{figure}
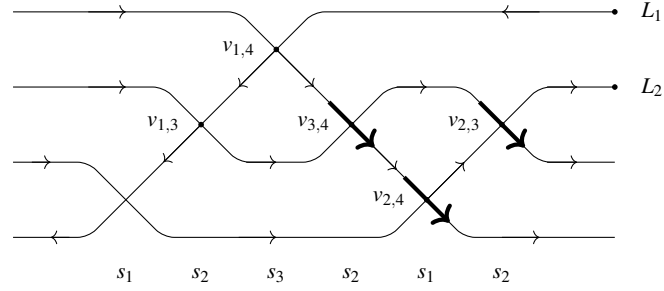

Back to the general case, we fix an orientation induced by $l_i, 1\le i <n$ and consider all rigorous paths from $L_{i}$ to $L_{i+1}$. We associate the \emph{weight} $c_{\bf p}$ to each such path ${\bf p}$ as follows. Denote by $\{c_{i,j}\}_{1\le i,j\le n}$ the standard basis of $\mathbb R^N$, where we set $c_{i,j}=-c_{j,i}$ if $i>j$ and $c_{j,j}=0$. Note that $N$ is the number of crossings in a pseudoline arrangement and hence we can associate the basis vector $c_{i,j}$ to the crossing of $l_i$ and $l_j$ for $1\le i,j\le n$. Consider a rigorous path ${\bf p}=L_{i}\to v_{r_1}\to\dots\to v_{r_m}\to L_{i+1}$. Every vertex $v_{r_s}$ corresponds to the crossing of two lines $l_k$ and $l_j$. If ${\bf p}$ changes from line $l_k$ to line $l_j$ at $v_{r_s}$ we associate the vector $c_{k,j}\in\mathbb R^{N}$. We set $c_{\bf p}$ to be the sum of all such $c_{k,j}$ in ${\bf p}$ and denote it by $c_{\bf p}$. 
 
\begin{definition}\label{def:stringcone}
For a fixed reduced expression ${\underline w_0}$, we define the \emph{string cone} to be 
\[
C_{\underline w_0} = \{ (y_{i,j}) \in \mathbb{R}^N \mid (c_{\bf p})^T(y_{i,j}) \geq 0, \forall \; {\bf p}\in\mathcal P_{\underline w_0}\}.
\]
\end{definition}
This is not the original definition of a string cone, 
but an equivalent one (see \cite[Corollary~5.8]{GP00}). 
It can be extended to describe string cones for Schubert varieties, see \cite{BF}.

\begin{example} There are two rigorous paths in Figure~\ref{fig:pseudo}, 
$L_1\to v_{1,3}\to v_{2,3}\to L_2$ and $L_1\to v_{1,3}\to v_{1,2} \to v_{2,3} \to L_2$.
The corresponding weights are $c_{1,3}-c_{2,3}$ and $c_{1,2}$ inducing the inequalities $y_{1,3}-y_{2,3}\ge 0$ 
and $y_{1,2}\ge 0$. Considering the orientation induced  by $l_2$ there is a rigorous path $L_2\to v_{2,3}\to L_3$ which gives the inequality $y_{2,3}\ge 0$. 
The string cone corresponding to the underlying non-oriented pseudoline arrangement in Figure~\ref{fig:pseudo} is 
then given by 
\[ 
C_{s_1s_2s_1}=\{ y_{1,2}\ge0,\ y_{1,3}\ge y_{2,3}\ge0  \}.
\]
\end{example}
 
Each crossing of lines $l_k$ and $l_m$ corresponds to an index $i_j$ associated to a simple reflection $s_{i_j}$ in $\underline w_0$ (see e.g. Figure~\ref{fig:pseudo}). We will therefore also denote $c_{k,m}=c_{j}$. Let $1\le i\le n-1$ and $r_1,\dots,r_{n_i}$ be the indices such that $s_{i_{r_p}}=s_i$ in $\underline w_0$ for $1\le p\le n_i$. Further, let $k_1,\dots,k_t$ be the positions where $s_{i_{k_m}}\in\{s_{i-1},s_{i+1}\}$ for $1\le m\le t$. In particular, $r_1,\dots, r_{n_i}$ are those positions inducing a crossing at level $i$ in the corresponding pseudoline arrangement. The following appears in \cite{Lit98}. 
 
 \begin{definition}\label{weightedcone}
 The \emph{weighted string cone} $\mathcal C_{\underline w_0}\subset \mathbb R^{N}\times \mathbb R^{n-1}_{\ge 0}$ is obtained from $ C_{\underline w_0}$ by adding variables $m_1,\dots, m_{n-1}$, and for every $1\le i\le n-1$ and $j\in\{r_1,\dots,r_{n_i}\}$ the inequality
 \begin{eqnarray*} 
 m_i-y_{j}-2\sum_{r_p>j} y_{r_{p}}+\sum_{k_p>j}y_{k_p}\ge 0,
\end{eqnarray*}
where $(y_{k},m_l)_{1\le k\le N\atop 1\le l< n}\in \mathbb R^{N}\times \mathbb R^{n-1}_{\ge 0
}$. For a weight $\lambda=\sum_{i=1}^{n-1}a_i\omega_i\in P$ the \emph{string polytope} is defined as 
\[ 
Q_{\underline w_0}(\lambda):=Q_{\underline w_0}\cap H_{\lambda}.
\]
Here $H_{\lambda}$ is the intersection of the hyperplanes defined by $m_i=a_i$ for all $1\le i< n$.
 \end{definition}

The additional \emph{weight inequalities} can also be obtained combinatorially as described in \cite{BF}. We will consider for all computations the weight $\rho=\sum_{i=1}^{n-1}\omega_i$. This is the weight in $P^{++}$ with minimal choice of coefficients of fundamental weights in $\mathbb {Z}_{>0}$, namely all are 1. Note that all string polytopes are cut out from the weighted string cone, but for different weights they are different polytopes.
 
\medskip
 
The following result is a simplified version of Theorem~1 proven by Caldero \cite{Cal02} for flag varieties. A more  general statement is given by Alexeev and Brion in \cite[Theorem~3.2]{AB04}.
\begin{theorem}\label{thm:AB}
There exists a flat family $\mathcal X\to \mathbb A^1$ for a normal variety $\mathcal X$ such that for $t\not =0$ the fiber over $t$ is isomorphic to $\Flag_n$ and for $t=0$ it is isomorphic to a projective toric variety $X_0$ with polytope $Q_{\underline w_0}(\lambda)$ for $\lambda \in P^{++}$. 
 \end{theorem}
The proof of Theorem~\ref{thm:AB} uses the embedding $\Flag_n \hookrightarrow \mathbb P(V(\lambda))$ and Lusztig's dual canonical basis, where $V(\lambda)$ is the irreducible representation of $\mathfrak{sl}_n$ with highest weight $\lambda$. 

\medskip
For two polytopes $A,B\subset \mathbb R^l$, the \emph{Minkowski sum} is defined to be $A+B=\{a+b: a\in A,b\in B\}$. Consider the weight $\rho$. The string polytope $Q_{\underline w_0}(\rho)$ is in general \emph{not} the Minkowski sum of string polytopes $Q_{\underline w_0}(\omega_1),\dots,Q_{\underline w_0}(\omega_{n-1})$, which motivates the following definition.

\begin{definition}\label{def:mp}
A string cone has the \emph{weak Minkowski property} (MP), if for every lattice point $p\in Q_{\underline w_0} (\rho)$ there exist lattice points $p_{\omega_i}\in Q_{\underline w_0}(\omega_i)$ such that
\[
p=p_{\omega_1}+p_{\omega_2}+\dots + p_{\omega_{n-1}}.
\]
\end{definition}

\begin{remark}
Note that the (non-weak) Minkowski property would require the above condition on lattice points to be true for arbitrary weights $\lambda$. Further, note that if $Q_{\underline w_0}(\rho)$ is the Minkowski sum of the fundamental string polytopes $Q_{\underline w_0}(\omega_i)$, then MP is satisfied.
\end{remark}

\begin{proposition}\label{thm:stringpoly}
For $\Flag_4$ there are four string polytopes in $\mathbb R^{10}$ up to unimodular equivalence 
and three of them satisfy MP. For $\Flag_5$ there are $28$ string polytopes in $\mathbb R^{14}$ up to unimodular equivalence and $14$ of them satisfy MP.
\end{proposition}

\begin{proof}
We first consider $\Flag_4$. There are 16 reduced expressions for $w_0$. Simple transpositions $s_i$ and $s_j$ with $1\le i<i+1<j< n$ commute and are also called \emph{orthogonal}. We consider reduced expressions up to changing those, i.e. there are eight symmetry classes. We fix the weight in $P^{++}$ to be $\rho=\omega_1+\omega_2+\omega_3$. The string polytopes are organized in four classes up to unimodular equivalence. See Table~\ref{tab:string4}, in which $121321$ denotes the reduced expression $\underline w_0=s_1s_2s_1s_3s_2s_1$. Hence they give four different toric degenerations for the embedding $\Flag_4\hookrightarrow \PP(V(\rho))$.

\begin{table}
\begin{tabular}{| l|  l|  l|  l|  l|  l|}

\hline
$\underline w_0$ & Normal & MP & Weight vector ${\bf w}_{\underline w_0}$& Prime  & Tropical cone \\
\hline

\begin{tabular}{ l}
String 1: \\
121321 \\
212321 \\
232123 \\
323123
\end{tabular} &
\begin{tabular}{ l}
 \\
yes \\
yes \\
yes \\
yes
\end{tabular}
&
\begin{tabular}{ l}
 \\
yes \\
yes \\
yes \\
yes
\end{tabular}
&
\begin{tabular}{ l}
 \\
$(0, 32, 24, 7, 0, 16, 6, 48, 38, 30, 0, 4, 20, 52)$ \\
$(0, 16, 48, 7, 0, 32, 6, 24, 22, 54, 0, 4, 36, 28)$ \\
$(0, 4, 36, 28, 0, 32, 24, 6, 22, 54, 0, 16, 48, 7)$ \\
$(0, 4, 20, 52, 0, 16, 48, 6, 38, 30, 0, 32, 24, 7)$
\end{tabular}
&
\begin{tabular}{ l}
 \\
yes \\
yes \\
yes \\
yes
\end{tabular} &
\begin{tabular}{ l}
 \\
rays 10, 18, 19, cone 71 \\
rays 6, 10, 19, cone 44 \\
rays 0, 3, 6, cone 3 \\
rays 0, 1, 3, cone 1
\end{tabular} \\
\hline

\begin{tabular}{ l}
String 2: \\
123212 \\
321232
\end{tabular} &
\begin{tabular}{ l}
  \\
yes \\
yes
\end{tabular}&
\begin{tabular}{ l}
  \\
yes \\
yes
\end{tabular}&
\begin{tabular}{ l}
  \\
$(0, 32, 18, 14, 0, 16, 12, 48, 44, 27, 0, 8, 24, 56)$ \\
$(0, 8, 24, 56, 0, 16, 48, 12, 44, 27, 0, 32, 18, 14)$
\end{tabular}
&
\begin{tabular}{ l}
  \\
yes \\
yes
\end{tabular}
&
\begin{tabular}{ l}
  \\
rays 2, 10, 18, cone 36 \\
rays 0, 1, 2, cone 0
\end{tabular} \\
\hline

\begin{tabular}{ l}
String 3: \\
213231
\end{tabular}
&\begin{tabular}{ l}
  \\
yes
\end{tabular}
&
\begin{tabular}{ l}
  \\
yes
\end{tabular}
&
\begin{tabular}{ l}
  \\
$(0, 16, 48, 13, 0, 32, 12, 20, 28, 60, 0, 8, 40, 22)$
\end{tabular}&
\begin{tabular}{ l}
  \\
yes
\end{tabular}
&
\begin{tabular}{ l}
  \\
rays 3, 6, 19, cone 24
\end{tabular}\\
\hline

\begin{tabular}{ l}
String 4: \\
132312 
\end{tabular}&
\begin{tabular}{ l}
  \\
yes
\end{tabular}
&
\begin{tabular}{ l}
  \\
no
\end{tabular}
&
\begin{tabular}{ l}
  \\
$(0, 16, 12, 44, 0, 8, 40, 24, 56, 15, 0, 32, 10, 26)$
\end{tabular}&
\begin{tabular}{ l}
  \\
no
\end{tabular}
&
\begin{tabular}{ l}
  \\
rays  1, 2, 17, cone 17
\end{tabular}\\
\hline
\ FFLV & \ yes & \ yes & 

\begin{tabular}{l }
$w^{min}=(0,2,2,1,0,1,1,2,1,2,0,1,1,1)$\\
$w^{reg}=(0,3,4,3,0,2,2,4,3,5,0,1,2,3)$
\end{tabular}
 & \ yes & 
\begin{tabular}{l}
rays 9, 11, 12, cone 56\\
rays 9, 11, 12, cone 56
\end{tabular}
 \\
\hline

\end{tabular}
\caption{Isomorphism classes of string polytopes for $n=4$ and $\rho$ depending on $\underline w_0$, normality, the weak Minkowsky property, the weight vectors ${\bf w}_{\underline w_0}$ constructed in \S\ref{string:weight}, primeness of the binomial initial ideals $\init_{{\bf w}_{\underline w_0}}(I_4)$, and the corresponding tropical cones with their spanning rays as they appear 
at \url{http://www.mi.uni-koeln.de/~lbossing/tropflag/tropflag4.html%https://github.com/ToricDegenerations
} .}  \label{tab:string4}
\end{table}

In order to verify whether the weak Minkowski property holds or not, we proceed as follows. We fix $\underline w_0$ to compute the string polytope $Q_{\underline w_0}(\rho)$ using \emph{polymake}. The number of lattice points in $Q_{\underline w_0}(\rho)$ is $\dim(V(\rho))=64$.
Then we compute the polytopes $Q_{\underline w_0}(\omega_1),Q_{\underline w_0}(\omega_2),Q_{\underline w_0}(\omega_3)$ and set $P=Q_{\underline w_0}(\omega_1)+Q_{\underline w_0}(\omega_2)+Q_{\underline w_0}(\omega_3)$. Now let $LP(P)$ be the set of lattice points in $P$. If $\vert LP(P)\vert <64$, then there exists a lattice point in $Q_{\underline w_0}(\rho)$, that can not be expressed as $p_1+p_2+p_3$ for $p_i\in Q_{\underline w_0}(\omega_i)$. For $\underline w_0=s_1s_3s_2s_3s_1s_2$ 
, we observe that
\[
\vert LP(Q_{\underline w_0}(\omega_1)+Q_{\underline w_0}(\omega_2)+Q_{\underline w_0}(\omega_3))\vert =62<64.
\]
Hence the class String 4 does not satisfy MP. For the classes String 1, 2, and 3 equality holds and MP is satisfied.

Now consider $\Flag_5$. There are 62 reduced expressions $\underline w_0$ up to changing orthogonal transpositions. The map $L:S_5\to S_5$ given on simple reflections by $L(s_i)=s_{4-i+1}$ induces a symmetry among the string polytopes. Namely, for a fixed $\lambda \in P^{++}$, there is a unimodular equivalence between $Q_{\underline w_0}(\lambda)$ and $Q_{L(\underline w_0)}(\lambda)$. Exploiting this symmetry, we compute 31 string polytopes for $\rho$. These are organized in 28 unimodular equivalence classes, that arise from further symmetries of the underlying pseudoline arrangements. Table~\ref{tab:stringweight5} shows which reduced expressions belong to string polytopes within one class of unimodular equivalence, and which string cones satisfy MP. Proceeding as for $\Flag_4$, we observe that 14 out of 28 classes satisfy MP.
\qed
\end{proof}

We will now turn to the FFLV polytope. It is defined in \cite{FFL11} by Feigin, Fourier, and Littelmann to describe bases of irreducible highest weight representations $V(\lambda)$. In \cite{FeFL16} they give a construction of a flat degeneration of the flag variety into the toric variety associated to the FFLV polytope. It is also an example of the more general setup presented in \cite{FFL15}. We give the general definition here and compute the FFLV polytopes for $\Flag_4$ and $\Flag_5$ for $\rho$. Recall, that $\alpha_i$ for $1\le i< n$ are the simple roots of $\mathfrak{sl}_n$, and  $\alpha_{p,q}$ is the positive root $\alpha_p+\alpha_{p+1}+\dots+\alpha_{q}$ for $1\le p\le q<n$.

\begin{definition}\label{def:dyck}
A \emph{Dyck path} is a sequence of positive roots 
${\bf d}=(\beta_0,\ldots,\beta_k)$ with $ k\ge0$
satisfying the following conditions
\begin{enumerate}
\item if $k=0$ then ${\bf d}=(\alpha_i)$ for $1\le i\le n-1$,
\item if $k\ge 1$ then \begin{enumerate}
    \item the first and the last roots are simple, i.e. $\beta_0=\alpha_i$, $\beta_k=\alpha_j$ for $1\le i<j\le n-1$, 
    \item if $\beta_s=\alpha_{p,q}$ then $\beta_{s+1}$ is either $\alpha_{p,q+1}$ or $\alpha_{p+1,q}$.
 \end{enumerate}
\end{enumerate}
Denote by $\mathcal D$ the set of all Dyck paths. We choose the positive roots $\alpha>0$ as an indexing set for a basis of $\mathbb R^N$.
\end{definition}

\begin{definition}\label{def:ffl}
The \emph{FFLV polytope} $P(\lambda)\subset \mathbb R^{N}_{\ge 0}$ for a weight $\lambda=\sum_{i=1}^{n-1}m_i\omega_i\in P^{++}$ is defined as
\begin{eqnarray*}
P(\lambda)=\left\{ (r_{\alpha})_{\alpha>0}\in \mathbb R^N_{\ge 0}\left|  \begin{matrix} \forall {\bf d}\in \mathcal D:\text{ if }\beta_0=\alpha_{i} \text{ and } \beta_k=\alpha_{j} \\
r_{\beta_0}+ \dots +r_{\beta_k}\le m_{i}+\dots+m_{j}\end{matrix}\right.\right\}.
\end{eqnarray*}
\end{definition}

\begin{example}
Consider $\Flag_4$. Then the Dyck paths are
\begin{eqnarray*}
&(\alpha_1),(\alpha_2),(\alpha_3),&\\
&(\alpha_1,\alpha_{1,2},\alpha_2),(\alpha_2,\alpha_{2,3},\alpha_3),&\\
&(\alpha_1,\alpha_{1,2},\alpha_{2},\alpha_{2,3},\alpha_3)
\text{ and } (\alpha_1,\alpha_{1,2},\alpha_{1,3},\alpha_{2,3},\alpha_3)& 
\end{eqnarray*} For our favorite choice of weight $\lambda=\rho=\omega_1+\omega_2+\omega_3$ we obtain the FFLV polytope
\begin{eqnarray*}
P(\rho)=\left\{ (r_{\alpha})_{\alpha>0}\left| \begin{matrix} 
r_{\alpha_1}\le 1,r_{\alpha_2}\le 1,r_{\alpha_3}\le 1,\\
r_{\alpha_1}+r_{\alpha_{1,2}}+r_{\alpha_2}\le 2, r_{\alpha_2}+r_{\alpha_{2,3}}+r_{\alpha_3}\le 2,\\
r_{\alpha_1}+r_{\alpha_{1,2}}+r_{\alpha_2}+r_{\alpha_{2,3}}+r_{\alpha_3}\le 3,\\
r_{\alpha_1}+r_{\alpha_{1,2}}+r_{\alpha_{1,3}}+r_{\alpha_{2,3}}+r_{\alpha_3}\le 3
\end{matrix}\right.\right\}\subset \mathbb R^{6}_{\ge 0}.
\end{eqnarray*}
\end{example}

The following is a corollary of \cite[Proposition~11.6]{FFL11}, which says that a strong version of the Minkowski property is satisfied by the FFLV polytope for $\Flag_n$. It can alternatively be shown for $n=4,5$ using the methods in the proof of Proposition~\ref{thm:stringpoly}.

\begin{corollary}
The FFLV polytope $P(\rho)$ satisfies the weak Minkowski property.
\end{corollary}

\begin{remark}
The FFLV polytope is in general not a string polytope. A computation in \emph{polymake} 
shows that $P(\rho)$ for $\Flag_5$ is not combinatorially equivalent to 
any string polytope for $\rho$.
\end{remark}

\section{String cones and the tropicalized flag variety}\label{string:weight}

We have seen in \S\ref{sec:background} how to obtain  toric degenerations from  maximal prime cones of the tropicalization of the flag varieties. We compare the different toric degenerations that arise from the different approaches. Moreover, applying \cite[Lemma~3.2]{Cal02} we construct a weight vector from a string cone. Computational evidence for $\Flag_4$ and $\Flag_5$ shows that each constructed weight vector lies in the relative interior of a maximal cone in $\trop(\Flag_n)$. A similar idea for a more general case is carried out in \cite[\S7]{KM16}. For the FFLV polytope we compute weight vectors for $\Flag_n$ with $n=4,5$ (see Example~\ref{exp.ffl}) following a construction given in \cite{FFR15}. 

\medskip
We will now prove the result in Theorem~\ref{ts-comparison} by analyzing the  polytopes associated to the different  toric degenerations of  $\Flag_n$ for $n=4,5$.

\begin{table}
\begin{center}
\begin{tabular}{| l |l|  l|  l| }
\hline
Orbit  & Combinatorially equivalent polytopes\\
\hline
1  &  String 2 \\
\hline
2 &  String 1 (Gelfand-Tsetlin)\\
\hline
3  & String 3 and FFLV \\
\hline
4 & - \\
\hline
\end{tabular}
\end{center}
\caption{Combinatorial equivalences among the polytopes obtained from prime cones in $\trop(\Flag_4)$ and string polytopes resp. the FFLV polytope.}\label{tab:f-vector4}
\end{table}

\begin{proof}[of Theorem~\ref{ts-comparison}]
In order to distinguish the different toric degenerations, we consider the toric varieties associated to their special fibers. In case of the degenerations induced by the string polytopes and FFLV polytope, these toric varieties are normal. This might  not be true for the degenerations found in Theorem \ref{flag4} and Theorem \ref{flag5}. Hence, we consider two toric degenerations to be different if the normalization of their special fibers are not isomorphic.

Two toric varieties are isomorphic, if their corresponding fans are unimodular equivalent. In our case  the fans are the normal fans of the polytopes.
For this reason we first look for combinatorial equivalences between those. If they are not combinatorially equivalent then their normal fans can not be unimodular equivalent. We use \emph{polymake}~\cite{GJ00} for computations with polytopes.

From Table~\ref{tab:f-vector4} one can see that for $\Flag_4$ there is one toric degeneration, whose associated polytope is not combinatorially 
equivalent to any string polytope or the FFLV polytope for $\rho$.
Hence, its corresponding normal toric variety is not isomorphic to any toric variety associated to these polytopes. For the toric varieties associated to the other polytopes we can not exclude isomorphism since there might be a unimodular equivalences between pairs of normal fans.

For $\Flag_5$, Table~\ref{tab:f-vector5} in the appendix shows that  there are 168 polytopes obtained from prime cones of $\trop(\Flag_5)$ that are not combinatorially equivalent to any string polytope or the FFLV polytope for $\rho$.\qed  
\end{proof}

\begin{remark}
There are also string polytopes, which are not combinatorially equivalent to any polytope from prime cones in $\trop(\Flag_n)$ for $n=4,5$. These are exactly those not satisfying MP, i.e. one string polytope for $\Flag_4$ and 14 for $\Flag_5$. See also Table~\ref{tab:stringweight5}.
\end{remark}

From now on, we fix a reduced expression $\underline w_0=s_{i_1}\ldots s_{i_N}$ and we consider the sequence of simple roots $S=(\alpha_{i_1},\dots, \alpha_{i_N})$. Recall that for a positive root $\alpha$ we denote by $f_{\alpha}$ the root vector in $\mathfrak n^-\subset \mathfrak{sl}_n$ of weight $-\alpha$. By \cite[Lemma~2]{FFL15} the following holds.
\begin{proposition}\label{univ.env}
The universal enveloping algebra $U(\mathfrak n^-)$ is linearly generated by monomials of the form ${\bf f^ m}=f_{\alpha_{i_1}}^{m_1}\ldots f_{\alpha_{i_N}}^{m_N}$ for $m_i\in \mathbb N$. 
\end{proposition}

The proposition may be interpreted as a definition of the universal enveloping algebra. Given a weight $\lambda$, the irreducible highest weight representation $V=V(\lambda)$ is cyclically generated by a highest weight vector $v_\lambda \in V(\lambda)$, i.e. $V(\lambda)=U(\mathfrak n^-).v_\lambda$.

\begin{example}\label{ex:exterioralg}
For $\Flag_4$ three root vectors in $\mathfrak n^-$ are 
\[
f_{\alpha_1}=\begin{bmatrix} 0 & 0 & 0 & 0\\ 1 & 0 & 0 & 0\\ 0 & 0 & 0 & 0\\ 0 & 0 & 0 & 0\end{bmatrix},\quad f_{\alpha_2}=\begin{bmatrix} 0 & 0 & 0 & 0\\ 0 & 0 & 0 & 0\\ 0 & 1 & 0 & 0\\ 0 & 0 & 0 & 0\end{bmatrix},\quad\text{ and }\quad f_{\alpha_3}=\begin{bmatrix} 0 & 0 & 0 & 0\\ 0 & 0 & 0 & 0\\ 0 & 0 & 0 & 0\\ 0 & 0 & 1 & 0\end{bmatrix}.
\]
Consider $V=\bigwedge^2\mathbb C^4$. The action of $\mathfrak n^-$ on $\mathbb C^4$ is given by $f_{\alpha_i}(e_i)=e_{i+1}$ and $f_{\alpha_i}(e_j)=0$ for $j\not = i$. 
On $V$ the $\mathfrak n^-$-action is given by
\[
f_{\alpha_i}(e_j\wedge e_k)=f_{\alpha_i}(e_j)\wedge e_k + e_{j}\wedge f_{\alpha_i}(e_k).
\]
Let $e_1\wedge e_3\in V$. Then $f_{\alpha_2}(e_1\wedge e_2)=e_1\wedge e_3$. In fact, $V=U(\mathfrak n^-).(e_1\wedge e_2)$, this implies that $e_1\wedge e_2=:v_{\omega_2}$ is a highest weight vector. If we fix $\underline w_0=s_1s_2s_1s_3s_2s_1$, we have $U(\mathfrak n^-)=\langle f_{\alpha_1}^{m_1}f_{\alpha_2}^{m_2}f_{\alpha_1}^{m_3}f_{\alpha_3}^{m_4}f_{\alpha_2}^{m_5}f_{\alpha_1}^{m_6}:m_i\in \mathbb N\rangle$. Hence
\[{\bf f}^{(0,1,0,0,0,0)}(e_1\wedge e_2)={\bf f}^{(0,0,0,0,1,0)}(e_1\wedge e_2).\] 
\end{example}

As seen in Example~\ref{ex:exterioralg}, the monomial $\bf f^m$ for a given weight vector $v\in V$ with ${\bf f^m}(v_\lambda)=v$ is not unique. To fix this, we define a term order on the monomials ${\bf f^m}$ generating $U(\mathfrak n^-)$ and pick the minimal monomial with this property. We fix for ${\bf m,n}\in \mathbb N^N$ the order
\begin{eqnarray*}\label{def:term}
{\bf f^m}\succ {\bf f^n},\text{ if }\deg({\bf f^m})>\deg({\bf f^n})\text{ or }\deg({\bf f^m})=\deg({\bf f^n})\text{ and }{\bf m}<_{lex}{\bf n}.
\end{eqnarray*}

The connection to $\trop(\Flag_n)$ is established through Pl\"ucker coordinates. For $J=\{j_1,\dots,j_k\}\subset [n]$, $p_J$ is given by the Pl\"ucker embedding as $(e_{j_1}\wedge \ldots \wedge e_{j_k})^*\in (\bigwedge^k\mathbb C^n)^*$, the dual vector space. Now $\bigwedge^k\mathbb C^n$ is the fundamental representation $V(\omega_k)=U(\mathfrak n^-).(e_1\wedge \ldots \wedge e_k)$ (see Example~\ref{ex:exterioralg}). Denote by ${\bf m}_J$ the unique multiexponent such that ${\bf f}^{{\bf m}_J}$ is $\prec$-minimal satisfying ${\bf f}^{\bf m}(e_1\wedge \ldots \wedge e_k)=e_{j_1}\wedge \ldots \wedge e_{j_k}$.

Following a construction given in \cite[Proof of Lemma 3.2]{Cal02}, we define the linear form $e:\mathbb N^N\to\mathbb N$ as $e({\bf m})=2^{N-1}m_1+2^{N-2}m_2+\ldots +2m_{N-1}+m_N$. This is a particular choice satisfying ${\bf m}\succ {\bf n} \Rightarrow e({\bf m})>e({\bf n})$ for ${\bf m,n}\in\mathbb N^N$.

\begin{definition}\label{str.wt.vec}
For a fixed reduced expression ${\underline w_0}$ the \emph{weight} of the Pl\"ucker variable $p_J$ is $e({\bf m}_J)$. We fix the \emph{weight vector} ${\bf w}_{\underline w_0}$ in $\mathbb R^{{n\choose 1}+{n\choose 2}+\cdots+{n\choose n-1}}$ to be 
\[
{\bf w}_{\underline w_0}=(e({\bf m}_1),e({\bf m}_2),\ldots,e({\bf m}_{2,3,\ldots,n})).
\]
\end{definition}

\begin{example}
We continue as in Example~\ref{ex:exterioralg} with the fixed reduced expression $\underline w_0=s_1s_2s_1s_3s_2s_1$ for $\Flag_4$. The Pl\"ucker coordinate $p_{13}$ in $\Flag_4$ is $(e_1\wedge e_3)^*$. The corresponding minimal monomial among those satisfying ${\bf f^m}(e_1\wedge e_2)=e_1\wedge e_3$ is ${\bf f}^{(0,1,0,0,0,0)}$. Hence the weight of $p_{13}$ is $e(0,1,0,0,0,0)=1\cdot 2^4=16$. Similarly, we obtain the weights of all Pl\"ucker coordinates and
\[
{\bf w}_{\underline w_0}=(0,32,24,7,0,16,6,48,38,30,0,4,20,52).
\]
Table~\ref{tab:string4} contains all weight vectors for $\Flag_4$ constructed in the way just described.
\end{example}

\begin{proposition}\label{thm:stringweight}
Consider $\Flag_n$ with $n=4,5$. The above construction produces a weight vector ${\bf w}_{\underline w_0}$ for every string cone. This weight vector lies in the relative interior of a maximal cone of $\trop(\Flag_n)$.
If further the string cone satisfies MP, then ${\bf w}_{\underline w_0}$ lies in the relative interior of a prime cone whose associated polytope is combinatorially equivalent to $Q_{\underline w_0}(\rho)$.
\end{proposition}

\begin{proof}
The constructed weight vectors ${\bf w}_{\underline w_0}$ can be found in Table~\ref{tab:string4} for $\Flag_4$ and Table~\ref{tab:stringweight5} in the appendix for $\Flag_5$. A computation in \emph{Macaulay2} shows that all initial ideals $\init_{{\bf w}_{\underline w_0}}(I_n)$ for $n=4,5$ are binomial, hence in the relative interiors of maximal cones of $\trop(\Flag_n)$.

Moreover, if MP is satisfied  we check using \emph{polymake} that  the polytope constructed from the maximal prime cone $C_{\underline w_0}$ with ${\textbf w}_{\underline w_0}$ in its relative interior is combinatorially equivalent to the string polytope $Q_{\underline w_0}(\rho)$. See Table~\ref{tab:string4} and Table~\ref{tab:stringweight5}. 
\qed
\end{proof}

Computational evidence leads us to the following conjecture.

\begin{conjecture}\label{conjecture}
Let $n\geq 3$ be an arbitrary integer. For every reduced expression $\underline w_0$, the weight vector ${\bf w}_{\underline w_0}$ lies 
in the relative interior of a maximal cone in $\trop(\Flag_n)$. 

In particular, if the string cone satisfies MP this vector lies in the relative interior of the prime cone $C$, whose associated polytope is combinatorially equivalent to the string polytope $Q_{\underline w_0}(\rho)$.
\end{conjecture}

The following example discusses a similar construction of weight vectors for the FFLV polytope.

\begin{example}\label{exp.ffl}
Consider for $\Flag_4$ the sequence of positive roots
\[
S=(\alpha_1+\alpha_2+\alpha_3,\alpha_1+\alpha_2,\alpha_2+\alpha_3,\alpha_1,\alpha_2,\alpha_3).
\]
By \cite[Example~1]{FFL15}, Proposition~\ref{univ.env} is also true for this choice of $S$. More generally speaking, Proposition~\ref{univ.env} holds for every sequence containing all positive roots ordered by height. The \emph{height} of a positive root is the number of simple summands. Such sequences are called \emph{PBW-sequences} with \emph{good ordering} in \cite{FFL15}.

With this choice of $S$ we apply the aformentioned procedure to obtain a unique multi-exponent ${\bf m}_J$ for each Pl\"ucker variable $p_J$. Taking the convex hull of all multi-exponents 
${\bf m}_J$ for $J\subset \{1,\ldots,4\}$ yields the FFLV polytope from Definition~\ref{def:ffl} with respect to the embedding $\Flag_4\hookrightarrow \mathbb P(V(\rho))$. Then we define linear forms
\begin{eqnarray*}
e^{\min}({{\bf m}_J})=m_1+2m_2+m_3+2m_4+m_5+m_6,\\
e^{\rm reg}({\bf m}_J)=3m_1+4m_2+2m_3+3m_4+2m_5+m_6,
\end{eqnarray*}
according to the degrees defined in \cite{FFR15}. 
We obtain in analogy to Definition~\ref{str.wt.vec} the corresponding weight vectors
\begin{eqnarray*}
{\bf w}^{\min}=(0,2,2,1,0,1,1,2,1,2,0,1,1,1),\\
{\bf w}^{\rm reg}=(0,3,4,3,0,2,2,4,3,5,0,1,2,3).
\end{eqnarray*}

A computation in $\emph{Macaulay2}$ shows that $\init_{{\bf w}^{\min}}(I_4)=\init_{{\bf w}^{\rm reg}}(I_4)$
is a binomial prime ideal. Hence ${\bf w}^{\min}$ and ${\bf w}^{\rm reg}$ lie in the relative interior of the same prime cone $C\subset \trop(\Flag_4)$. Using \emph{polymake}~\cite{GJ00} we verify that the polytope associated to $C$ is combinatorially equivalent to the FFLV polytope $P(\rho)$. We did the analogue of this computation for $\Flag_5$ and the outcome is the same, $\init_{{\bf w}^{\min}}(I_5)=\init_{{\bf w}^{\rm reg}}(I_5)=\init_C(I_5)$ with the polytope associated to $C$ being combinatorially equivalent to $P(\rho)$. The weight vectors ${\bf w}^{min}$ and ${\bf w}^{reg}$ for $\Flag_5$ can be found in Table~\ref{tab:stringweight5} in the appendix.
\end{example}

\section{Toric degenerations from non-prime cones}\label{Algorithmic approach}

As we have seen in \S\ref{sec:3}, not all maximal cones in the tropicalization of a variety give rise to prime initial ideals and hence to toric degenerations. In fact, there may also be tropicalizations without prime cones (see Example \ref{ex:badCone}).
Let $X$ be a subvariety of a toric variety $Y$. 
In this section, we give a recursive procedure (Procedure \ref{alg:Kh_Basis}) to compute a new embedding $X'$ of $X$ in case $\trop(X)$ has non-prime cones. Let $C$ be a non-prime cone. If the algorithm terminates, the new variety $X'$ has more prime cones than $\trop(X)$ and at least one of them is projecting onto $C$. 
We apply this procedure to $\Flag_4$ and compare the new toric degenerations with those obtained so far (see Proposition~\ref{prop:output}). The procedure terminates for $\Flag_4$, but we are still investigating the conditions for which this is true in general.

\medskip

\begin{algorithm}[h]
\SetAlgorithmName{Procedure}{} 
\KwIn{\medskip {\bf Input:\ }  $A = \mathbb C[x_0,\ldots,x_n]/I$, where $\mathbb C[x_0,\ldots,x_n]$ is the total coordinate ring of the toric variety $Y$ and $I $ defines the subvariety $V(I)\subset Y$,
$C$ a non-prime cone of $\trop(V(I))$.
}
\BlankLine

{\bf Initialization:}\\
Compute the primary decomposition of $\init_C(I)$;\\
$I(W_C)=$ unique prime toric component in the decomposition;\\
$G=$ minimal generating set of $I(W_C)$.\\
Compute a list of binomials $L_C=\{f_{1},\ldots,f_{s}\}$ in $G$, which are not in $\init_C(I)$;\\
$A'=\mathbb C[x_0, \dots, x_n,y_{1},\ldots,y_{s}]/I'$ with $I'=I+\langle y_{1}-f_{1},\ldots,y_{s}-f_{s}\rangle $;\\
$V(I')$ subvariety of $Y'$ whose total coordinate ring is $\mathbb C[Y]:=\mathbb C[x_0, \dots, x_n,y_{1},\ldots,y_{s}]$.\\
Compute $\trop(V(I'))$;\\
\For {all  prime cones $C'\in\trop(V(I'))$} {

\If{    $\pi(C')$ is contained in the relative interior of $C$ }{{\bf Output:} The algebra $A'$   
and the ideal $\init_{C'}(I')$ of a  toric degeneration of $V(I')$.}\Else{
Apply the procedure again  to $A'$  and $C'$.}
}
\label{alg:Kh_Basis}
\caption{Computing  new embeddings of  the variety $X$ in case $\trop(X)$ contains non-prime cones}
\end{algorithm}

We now explain Procedure~\ref{alg:Kh_Basis}. Consider a toric variety $Y$ whose total coordinate ring is $\mathbb C[x_0, \dots, x_n]$ with associated $\mathbb Z^k$-degree $\deg:\mathbb Z^{n+1}\to\mathbb Z^k$. Let $X$ be the subvariety of $Y$ associated to an ideal $I\subset \mathbb C[x_0, \dots, x_n]$, where the Krull dimension of $A=\mathbb C[x_0, \dots, x_n]/I$ is $d$. Denote by $\trop(V(I))$ the tropicalization of $X$ intersected with the torus of $Y$. Suppose there is a non-prime cone $C\subset \trop(V(I))$ with multiplicity one. By  Lemma \ref{toric:lemma}, we have that   $I(W_C)$ is the unique  toric ideal in the primary decomposition of $\init_{C}(I)$, hence $\init_{C}(I)\subset I(W_C)$. We can compute $I(W_C)$ using the function  $\mathtt{primaryDecomposition}$ in \emph{Macaulay2}.
Fix a minimal binomial generating set $G$ of $I(W_C)$, and let $L_C=\{f_1, \ldots, f_{s}\}$ be the set consisting of binomials in $G$, which are not in $\init_{C}(I)$. 
By Hilbert's Basis Theorem, ${s}$ is a finite number. The absence of these binomials in $\init_C(I)$ is the reason why the initial ideal is not equal to $I(W_C)$.
We introduce new variables $\{y_{1},\ldots,y_{{s}}\}$  and consider the algebra $A' = \mathbb C[x_0, \ldots, x_n,y_{1},\ldots,y_{s}]/I'$, where 
\[
I'=I+\langle y_{1}-f_{1},\ldots,y_{s}-f_{s}\rangle.
\]  
The ideal $I'$  is a homogeneous ideal in $\mathbb{C}[x_0, \ldots, x_n, y_{1},\ldots,y_{s}]$ graded by  
\[
(\deg(x_0), \ldots, \deg(x_n),\deg(f_1),\ldots,\deg(f_s)).
\]
The new variety $V(I')$ is a subvariety of a toric variety $Y'$, which has total coordinate  $\mathbb C[Y']:=\mathbb C[x_0, \ldots, x_n,y_{1},\ldots,y_{s}]$.
For example, if $V(I)$ is a subvariety of a projective space then $V(I')$ is contained in a weighted projective space.

Since the algebras $A$ and $A'$ are isomorphic as graded algebras, the varieties $V(I)$ and $V(I')$ are isomorphic. 
We have a monomial map 
\[
\pi: \mathbb C[x_0,\ldots,x_n]/I\to \mathbb C[x_0, \ldots, x_n,y_{1},\ldots,y_{s}]/I'
\]
inducing a surjective map $\trop(\pi):\trop(V(I'))\to \trop(V(I))$ (see \cite[Corollary 3.2.13]{M-S}). The map $\trop(\pi)$  is the projection onto the first $n$ coordinates.
Suppose there exists a prime cone $C'\subset \trop(V(I'))$, whose projection has a non-empty intersection with the relative interior of $C$. Then by construction we have $\init_C(I)\subset \init_{C'}(I') \cap \mathbb{C}[x_0,\ldots,x_n]$ and the procedure terminates. In this way we obtain a new initial ideal $\init_{C'}(I')$ which is toric and hence gives a new toric degeneration of the variety $V(I')\cong V(I)$.
If only non-prime cones are projecting to $C$ then run this procedure again with $A'$ and $ C'$, where the latter is a maximal cone of $\trop(V(I'))$, which projects to $C$.
We can then repeat the procedure starting from a different non-prime cone.

The  function to apply Procedure \ref{alg:Kh_Basis} is  \texttt{findNewToricDegenerations} and it is part of the package $\mathtt{ToricDegenerations}$. This will compute only one re-embedding for each non-prime cone. It is possible to use \texttt{mapMaximalCones} to obtain the image of $\trop(V(I'))$ under the map $\pi$.

\begin{remark}
If $f_i$ is a polynomial in $\Bbbk[x_0,x_1,\ldots,x_n]$ with the standard grading and  $ \deg(f_i) > 1$, then we need to
homogenize the ideal $I'$ before computing the tropicalization with \emph{Gfan}. This is done by adding a new variable $h$. The homogenization of $I'$ with respect to $h$ is denoted by $I'_{proj}\subseteq \Bbbk[x_0,\ldots,x_n,y_1,\ldots,y_s,h]$. Then by \cite[Proposition 2.6.1]{M-S} for every $\bf {w}$ in $\mathbb {R}^{n+s+2}$ the ideal $\init_{\bf w}(I')$ is obtained from $\init_{(\bf{w},\text{0})}(I'_{proj})$ by setting $h=1$.
\end{remark}

If the cone $C$ is prime, we can construct a valuation $\val_C$ on $\Bbbk[x_0,\ldots,x_n]/I$ in the following way. Consider the matrix $W_C$ in Equation~(\ref{def:W}). For monomials $m_i = c {\bf x} ^{{\bf {\alpha}}_i}\in \Bbbk[x_0,\ldots x_n] $ define \begin{equation}\val(m_i) = W_C {\bf \alpha}_i\quad \text{and}\quad \val (\sum_i m_i) = \min_i \{\val(m_i)\},\end{equation} where the minimum on the right side is taken with respect to the lexicographic order on $(\ZZ^d, +)$. This is a valuation on $\Bbbk[x_0,\dots,x_n]$ of rank equal to the Krull dimension of $A$ for every cone $C$. 
Composing $\val $ with the quotient morphism $p:\Bbbk[x_0,\ldots,x_n]\to \Bbbk[x_0,\ldots,x_n]/I$ we obtain a map $\val_C$, which is a valuation if and only if the cone $C$ is prime. Moreover, in \cite{KM16} Kaveh and Manon prove that a cone $C$ in $\trop(V(I))$ is prime if and only if $A=\Bbbk[x_0,\ldots,x_n]/I$ has a  finite \textit{Khovanskii basis} for the valuation $\val_C$ constructed from the cone $C$. Recall that a Khovanskii basis for an algebra $A$ with valuation $\val_C$ is a subset $B$ of $A$ such that $\val_C(B)$ generates the value semigroup $S(A, \val_C)=\{\val_C(f): f \in A\setminus\{0\}\}$.

\medskip

Procedure \ref{alg:Kh_Basis} can be interpreted as finding an extension $\val_{C'}$ of $\val_C$  so that $A'$ has finite Khovanskii basis with respect to $\val_{C'}$. The  Khovanskii basis is given by the images of $x_0, \dots, x_n,y_1, \dots ,y_s$ in $A'$. 
We illustrate the procedure in  the following  example. 

\begin{example}
\label{ex:badCone}
Consider the algebra $A = \mathbb C[x, y, z]/ \langle xy+xz+yz\rangle $. The tropicalization of $V(\langle xy+xz+yz\rangle)\subset \mathbb P^2$ has three maximal cones. The corresponding initial ideals are $\langle xz+yz\rangle,\langle xy+yz\rangle$ and $\langle xy+xz\rangle$, none of which is prime. Hence they do not give rise to toric degenerations. The matrices associated to each cone are
\[
{W_{C_1}} = \begin{bmatrix} 
0 &\ \ 0 & -1 \\
1 &\ \ 1 &\ \ 1  
\end{bmatrix},\quad
{W_{C_2}} = \begin{bmatrix} 
0 &-1 & \ \ 0 \\
1 &\ \ 1 &\ \ 1  
\end{bmatrix}\quad\text{and}\quad
{W_{C_3}} = \begin{bmatrix} 
-1 &\ \ 0 & \ \ 0 \\
\ \ \ 1 & \ 1 & \ 1  
\end{bmatrix}.
\]
We now apply Procedure \ref{alg:Kh_Basis} to the cone $C_1$. The initial ideal associated to $C_1$ is generated by $xz+yz$. In this case $\init_{C_1}(I)=\langle z \rangle \cdot  \langle x+y\rangle$ hence for the missing binomial $x+y$ we adjoin a new variable $u$ to $ \mathbb C[x, y, z]$  and the new relation $u-x-y$ to $I$. 
We have  
\[
I'=\langle xy+xz+yz,u-x-y\rangle\ \text{and}\ A' = \mathbb C[x, y,z,u]/I'
\]
with $V(I')$ a subvariety of $\mathbb P^3$.
After computing the tropicalization of $V(I')$ we see that there exists a prime cone $C'$ such that  $\pi(C')=C$. The matrix $W_{C'}$ associated to the cone $C'$ is
\[
{W'} = \begin{bmatrix} 
0 &\ \ 0 & -1 &\  1\\
1 &\ \ 1 &\ \ 1 &\ 1
\end{bmatrix}.
\]
The initial ideal $\init_{C'}(I')$ gives a toric degeneration of $V(I')$. The image of the set $\{ x, y, z,u\}$ in $A'$ is a Khovanskii basis for $S(A', \val_{C'})$. Repeating this process for the cones $C_2$ and $C_3$ of $\trop(V(xy+xz+yz))$, we get  prime cones $C'_2$ and $C'_3$ whose projections  are $C_2$ and $C_3$ respectively. 
Hence there is a valuation with finite Khovanskii basis and a corresponding toric degeneration for every maximal cone.
\end{example}

We now apply Procedure~\ref{alg:Kh_Basis}  to $\trop(\Flag_4)$.

\begin{proposition}\label{prop:output}
Each of the non-prime cones of $\trop(\Flag_4)$ gives rise to three toric degenerations, which are not isomorphic to any degeneration coming from the prime cones of $\trop(\Flag_4)$. Moreover, two of the three new polytopes are combinatorially equivalent to the previously missing string polytopes for $\rho$ in the class String 4.

\end{proposition}

\begin{proof}
By Theorem \ref{flag4} we know that $\trop(\Flag_4)$ has six non-prime cones forming one $S_4\rtimes\mathbb Z_2$-orbit. Hence, we apply Procedure~\ref{alg:Kh_Basis} to only one non-prime cone. The result for the other non-prime cones will be the same up to symmetry. In particular, the obtained toric degenerations from one cone will be isomorphic to those coming from another cone. 
We describe the results for the maximal cone $C$ associated to the  initial ideal $\init_C(I_4)$ defined by the following binomials: 
\begin{center}
\begin{tabular}{llll}

$p_4p_{123}-p_{3}p_{124},$ &$p_{24}p_{134}-p_{14}p_{234}$, &$p_{23}p_{134}-p_{13}p_{234}$,
&$p_{2}p_{14}-p_{1}p_{24}$, \\ $p_{2}p_{13}-p_{1}p_{23}$, &$p_{24}p_{123}-p_{23}p_{124}$,& $p_{14}p_{123}-p_{13}p_{124}$, &$p_{4}p_{23}-p_{3}p_{24}$\\ 
&$p_{4}p_{13}-p_{3}p_{14}$, \text{and}&$p_{14}p_{23}-p_{13}p_{24}.$
\end{tabular}
\end{center}
We define the ideal $I'=I_4+\langle w-p_{2}p_{134}+p_{1}p_{234}\rangle$.  
The grading on the variables $p_1,\ldots,p_{234}$ and $w$  
is given by the matrix 
\setcounter{MaxMatrixCols}{15}
\[
\begin{bmatrix}
 1 & 1 & 1 & 1 & 0 & 0 & 0 & 0 & 0 & 0 & 0 & 0 & 0 & 0 &1\\
 0 & 0& 0 &0 &1 &1&1&1&1&1&0&0&0&0&0\\
 0&0&0&0&0&0&0&0&0&0&1&1&1&1&1
\end{bmatrix}.
\]
It extends the grading on the variables $p_1,\ldots,p_{234}$ given by the matrix $D$ in (\ref{degree}).
The tropical variety $\trop(V(I'))$ has $105$ maximal cones, $99$ of which are prime. Among them we can find three maximal prime cones, which are mapped to $C$ by $\trop(\pi)$ (see Figure \ref{figure:2}). We compute the polytopes associated to the normalization of  these three toric degenerations by applying the same methods as in Theorem \ref{flag4}. Using \emph{polymake} we check that two of them are combinatorially equivalent to the string polytopes for $\rho$ in the class String 4. Moreover, none of them is combinatorially equivalent to any polytope coming from prime cones of $\trop(\Flag_4)$, hence they define different toric degenerations. \qed
\end{proof}

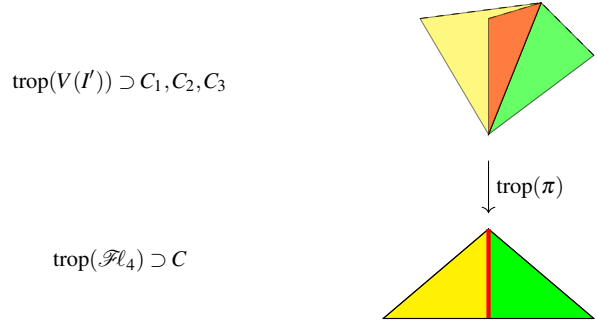
\begin{figure}
\begin{center}
\begin{tikzpicture}[scale=.7]
\node at (-4,4.5) {$\trop(V(I'))\supset C_1,C_2,C_3$};

\draw[-][dashed] (1.7,5.7) -- (4,6);
\draw[-][dashed] (3,5.7) -- (4,6);
\draw[-][dashed] (5,5) -- (4,6);
\draw[-] (3,3.5) -- (4,6);
\draw[fill= green,opacity=0.6] (3,3.5) -- (4,6) -- (5,5) --cycle;

\draw[fill= yellow,opacity=0.5] (1.7,5.7) -- (4,6) -- (3,3.5) --cycle;
\draw[fill= red,opacity=0.5] (3,5.7) -- (4,6) -- (3,3.5) --cycle;
\draw[->] (3,3) -- (3,2);
\node[right] at (3,2.5) {$\trop(\pi)$};
\node at (-4,1.1) {$\trop(\Flag_4)\supset C$};
\draw[fill=yellow] (1,0) -- (3,1.7) -- (3,0) --cycle;
\draw[fill= green] (5,0) -- (3,1.7) -- (3,0) --cycle;
\draw[-] (1,0) -- (3,1.7);
 \draw[-] (5,0) -- (3,1.7);
 \draw[-][ultra thick ,red] (3,0) -- (3,1.7);
\draw[-] (1,0) -- (5,0);
    \end{tikzpicture}
\caption{The three triangles above represent  the three cones in $\trop(V(I'))$ which project down to the non-prime cone $C$ in $\trop(\Flag_4)$.}\label{figure:2}
\end{center}
\end{figure}

\begin{remark}
Procedure~\ref{alg:Kh_Basis} could be applied also to $\Flag_5$ but we have not been able to do so. In fact, the tropicalization for $\trop(V(I_5'))$ did not terminate since the computation can not be simplified by symmetries.  \end{remark}

\begin{acknowledgement}
This article was initiated during the Apprenticeship
 Weeks (22 August-2 September 2016), led by Bernd Sturmfels, as part of the Combinatorial Algebraic Geometry Semester at the Fields Institute. The authors are  grateful to the Max Planck Institute MiS Leipzig, where  part of this project was carried out. We are grateful to Diane Maclagan, Kiumars Kaveh, and Kristin Shaw for inspiring conversations. We also would like to thank Diane Maclagan, Yue Ren and five anonymous referees for their comments on an earlier version of this manuscript. Further, L.B. and F.M. would like to thank Ghislain Fourier and Xin Fang for many inspiring discussions. K.M would like to express her gratitude to D\'aniel Jo\'o for many helpful conversations. F.M. was supported by a postdoctoral fellowship from the Einstein Foundation Berlin. S.L. was supported by EPSRC grant 1499803.
\end{acknowledgement}

\bibliographystyle{plain} 
\bibliography{Trop.bib}

\section*{Appendix}

In this Appendix we provide numerical evidence of our computations.  Table~\ref{nonprime}  contains data on the non-prime maximal cones of $\trop(\Flag_5)$. In Table~\ref{tab:f-vector5} there is information on the polytopes obtained from maximal prime cones of $\trop(\Flag_5)$. This includes the F-vectors, combinatiral equivalences among the polytopes, and between those and the string polytopes, resp. FFLV polytope, for $\rho$.
Lastly Table~\ref{tab:stringweight5} contains information on the string polytopes and FFLV polytope for $\Flag_5$, such as the weight vectors constructed in \S\ref{string:weight}, primeness of the initial ideals with respect to these vectors, and the MP property.  

\section*{Algebraic and combinatorial invariants of $\trop(\Flag_5)$}\label{polyFlag5}

Below we collect in a table all the information about the non-prime initial ideals of $\Flag_5$ up to symmetry.

\begin{table}
\begin{center}
\begin{tabular}{|l|l|l|l|l|l|}
 \hline
Number of Orbits & $\#$Generators\\
 \hline
30&   69\\
267  &66\\
37  &68\\
11  &70\\
10  &71\\
2&  73\\
\hline
\end{tabular}
\caption{Data for non-prime initial ideals of $\Flag_5$.}\label{nonprime}
\end{center}
\end{table}

The following table shows the F-vectors of the polytopes associated to maximal prime cones of $\trop(\Flag_5)$ for one representative in each orbit. The last column contains information on the existence of a combinatorial equivalence between these polytopes and the string polytopes resp. FFLV polytope for $\rho$. The initial ideals are all \textit{Cohen-Macaulay}.

\begin{center}
\begin{longtable}{| l | l | l |}
\hline
Orbit & F-vector & Combinatorial equivalences \\
\hline \endhead
0  &  475 2956 8417 14241 15690 11643 5820 1899 374 37 & \\ 
\hline
1  &  456 2799 7843 13023 14038 10159 4938 1565 301 30 & \\ 
\hline
2  &  425 2573 7108 11626 12333 8779 4201 1316 253 26 & \\ 
\hline
3  &  393 2313 6200 9833 10125 7021 3297 1027 201 22 & \\ 
\hline
4  &  433 2621 7230 11796 12473 8847 4219 1318 253 26 & \\ 
\hline
5  &  435 2630 7246 11810 12479 8848 4219 1318 253 26 & \\ 
\hline
6  &  425 2553 6988 11317 11888 8388 3987 1245 240 25 & \\ 
\hline
7  &  450 2751 7677 12699 13648 9863 4800 1529 297 30 & \\ 
\hline
8  &  435 2630 7246 11810 12479 8848 4219 1318 253 26 & \\ 
\hline
9  &  419 2522 6922 11243 11842 8373 3985 1245 240 25 & \\ 
\hline
10  &  453 2785 7817 12999 14027 10157 4938 1565 301 30 & \\ 
\hline
11  &  463 2885 8237 13987 15474 11532 5788 1895 374 37 & \\ 
\hline
12  &  463 2852 8020 13365 14459 10501 5121 1627 313 31 & \\ 
\hline
13  &  457 2840 8078 13638 14954 10996 5413 1726 330 32 & \\ 
\hline
14  &  454 2819 8016 13540 14870 10968 5427 1744 337 33 & \\ 
\hline
15  &  445 2748 7770 13050 14254 10464 5161 1658 322 32 & \\ 
\hline
16  &  441 2681 7438 12228 13056 9369 4525 1430 276 28 & \\ 
\hline
17  &  440 2704 7602 12684 13752 10014 4897 1560 301 30 & \\ 
\hline
18  &  471 2923 8298 13995 15369 11369 5667 1845 363 36 & \\ 
\hline
19  &  464 2883 8200 13861 15258 11313 5651 1843 363 36 & \\ 
\hline
20  &  467 2911 8309 14097 15574 11586 5804 1897 374 37 & \\ 
\hline
21  &  461 2876 8225 13993 15509 11575 5814 1903 375 37 & \\ 
\hline
22  &  397 2363 6416 10313 10755 7536 3561 1109 215 23 & \\ 
\hline
23  &  437 2669 7447 12319 13236 9556 4642 1475 286 29 & \\ 
\hline
24  &  425 2553 6988 11317 11888 8388 3987 1245 240 25 & \\ 
\hline
25  &  415 2498 6861 11158 11772 8339 3976 1244 240 25 & \\ 
\hline
26  &  470 2942 8436 14377 15944 11889 5955 1939 379 37 & \\ 
\hline
27  &  460 2856 8109 13656 14929 10944 5374 1712 328 32 & \\ 
\hline
28  &  449 2741 7634 12594 13487 9702 4695 1486 287 29 & \\ 
\hline
29  &  427 2592 7181 11778 12523 8926 4270 1334 255 26 & \\ 
\hline
30  &  425 2573 7108 11626 12333 8779 4201 1316 253 26 & FFLV \\ 
\hline
31  &  443 2708 7557 12495 13411 9667 4686 1485 287 29 & \\ 
\hline
32  &  397 2363 6416 10313 10755 7536 3561 1109 215 23 &  S22\\ 
\hline
33  &  425 2553 6988 11317 11888 8388 3987 1245 240 25 & \\ 
\hline
34  &  419 2522 6922 11243 11842 8373 3985 1245 240 25 & \\ 
\hline
35  &  405 2407 6518 10442 10851 7578 3571 1110 215 23 & \\ 
\hline
36  &  401 2387 6477 10398 10825 7570 3570 1110 215 23 & \\ 
\hline
37  &  368 2154 5755 9111 9373 6497 3052 953 188 21 & S21\\ 
\hline
38  &  379 2214 5892 9280 9494 6547 3063 954 188 21 & S27, S28\\ 
\hline
39  &  393 2313 6200 9833 10125 7021 3297 1027 201 22 & \\ 
\hline
40  &  358 2069 5453 8516 8653 5941 2778 870 174 20 & S1, S18, S26, S29 (Gelfand-Tsetlin)\\ 
\hline
41  &  459 2851 8111 13720 15118 11223 5614 1834 362 36 & \\ 
\hline
42  &  467 2913 8322 14133 15629 11636 5831 1905 375 37 & \\ 
\hline
43  &  423 2562 7083 11596 12313 8772 4200 1316 253 26 & \\ 
\hline
44  &  425 2573 7108 11626 12333 8779 4201 1316 253 26 & S24 \\ 
\hline
45  &  397 2363 6416 10313 10755 7536 3561 1109 215 23 & S23 \\ 
\hline
46  &  461 2876 8225 13993 15509 11575 5814 1903 375 37 & \\ 
\hline
47  &  400 2366 6377 10175 10546 7363 3480 1089 213 23 & \\ 
\hline
48  &  393 2313 6200 9833 10125 7021 3297 1027 201 22 & \\ 
\hline
49  &  393 2313 6200 9833 10125 7021 3297 1027 201 22 & \\ 
\hline
50  &  379 2214 5892 9280 9494 6547 3063 954 188 21 & S2, S19\\ 
\hline
51  &  426 2599 7257 12034 12981 9420 4602 1470 286 29 & \\ 
\hline
52  &  428 2594 7176 11761 12514 8947 4307 1359 263 27 & \\ 
\hline
53  &  419 2522 6922 11243 11842 8373 3985 1245 240 25 & \\ 
\hline
54  &  466 2917 8371 14288 15879 11870 5960 1944 380 37 & \\ 
\hline
55  &  443 2729 7692 12867 13982 10197 4987 1585 304 30 & \\ 
\hline
56  &  453 2787 7826 13011 14021 10122 4895 1539 293 29 & \\ 
\hline
57  &  469 2926 8358 14188 15679 11663 5839 1906 375 37 & \\ 
\hline
58  &  458 2825 7958 13286 14398 10472 5113 1626 313 31 & \\ 
\hline
59  &  472 2949 8435 14335 15854 11796 5902 1923 377 37 & \\ 
\hline
60  &  440 2704 7602 12684 13752 10014 4897 1560 301 30 & \\ 
\hline
61  &  472 2967 8561 14720 16525 12526 6410 2144 432 43 & \\ 
\hline
62  &  457 2842 8099 13726 15153 11266 5640 1842 363 36 & \\ 
\hline
63  &  465 2902 8296 14096 15588 11594 5795 1884 368 36 & \\ 
\hline
64  &  459 2851 8111 13720 15118 11223 5614 1834 362 36 & \\ 
\hline
65  &  428 2608 7269 12028 12946 9377 4576 1462 285 29 & \\ 
\hline
66  &  441 2681 7438 12228 13056 9369 4525 1430 276 28 & \\ 
\hline
67  &  418 2510 6876 11157 11753 8321 3969 1243 240 25 & \\ 
\hline
68  &  406 2442 6713 10943 11587 8245 3950 1241 240 25 & \\ 
\hline
69  &  373 2199 5926 9474 9849 6897 3267 1024 201 22 & \\ 
\hline
70  &  427 2586 7144 11681 12383 8806 4209 1317 253 26 & \\ 
\hline
71  &  451 2781 7840 13111 14243 10390 5089 1623 313 31 & \\ 
\hline
72  &  440 2704 7602 12684 13752 10014 4897 1560 301 30 & \\ 
\hline
73  &  406 2442 6713 10943 11587 8245 3950 1241 240 25 & \\ 
\hline
74  &  448 2764 7800 13061 14208 10377 5087 1623 313 31 & \\ 
\hline
75  &  462 2873 8181 13846 15258 11321 5656 1844 363 36 & \\ 
\hline
76  &  457 2842 8099 13726 15153 11266 5640 1842 363 36 & \\ 
\hline
77  &  469 2927 8364 14203 15699 11678 5845 1907 375 37 & \\ 
\hline
78  &  454 2802 7903 13216 14348 10453 5110 1626 313 31 & \\ 
\hline
79  &  451 2787 7879 13221 14419 10565 5200 1667 323 32 & \\ 
\hline
80  &  441 2705 7584 12611 13622 9885 4823 1537 298 30 & \\ 
\hline
81  &  454 2803 7914 13263 14455 10598 5231 1687 330 33 & \\ 
\hline
82  &  441 2697 7532 12465 13391 9660 4685 1485 287 29 & \\ 
\hline
83  &  445 2721 7593 12550 13461 9694 4694 1486 287 29 & \\ 
\hline
84  &  441 2697 7532 12465 13391 9660 4685 1485 287 29 & \\ 
\hline
85  &  445 2725 7617 12611 13546 9764 4728 1495 288 29 & \\ 
\hline
86  &  397 2363 6416 10313 10755 7536 3561 1109 215 23 & \\ 
\hline
87  &  368 2154 5755 9111 9373 6497 3052 953 188 21 & S5, S31 \\ 
\hline
88  &  452 2801 7946 13385 14654 10771 5309 1699 327 32 & \\ 
\hline
89  &  430 2624 7318 12097 12974 9329 4497 1411 269 27 & \\ 
\hline
90  &  456 2834 8071 13670 15083 11210 5612 1834 362 36 & \\ 
\hline
91  &  432 2633 7332 12104 12975 9341 4521 1430 276 28 & \\ 
\hline
92  &  467 2919 8359 14230 15769 11756 5892 1922 377 37 & \\ 
\hline
93  &  456 2834 8071 13670 15083 11210 5612 1834 362 36 & \\ 
\hline
94  &  426 2597 7244 11998 12926 9370 4575 1462 285 29 & \\ 
\hline
95  &  440 2708 7630 12769 13898 10169 5001 1603 311 31 & \\ 
\hline
96  &  432 2633 7332 12104 12975 9341 4521 1430 276 28 & \\ 
\hline
97  &  412 2479 6810 11083 11707 8306 3967 1243 240 25 & \\ 
\hline
98  &  415 2511 6945 11391 12133 8679 4174 1313 253 26 & \\ 
\hline
99  &  458 2845 8092 13676 15042 11132 5543 1800 353 35 & \\ 
\hline
100  &  437 2669 7447 12319 13236 9556 4642 1475 286 29 & \\ 
\hline
101  &  441 2703 7569 12562 13531 9780 4746 1502 289 29 & \\ 
\hline
102  &  427 2586 7144 11681 12383 8806 4209 1317 253 26 & \\ 
\hline
103  &  419 2522 6922 11243 11842 8373 3985 1245 240 25 & \\ 
\hline
104  &  437 2669 7447 12319 13236 9556 4642 1475 286 29 & \\ 
\hline
105  &  411 2470 6776 11012 11617 8235 3933 1234 239 25 & \\ 
\hline
106  &  413 2483 6808 11043 11606 8177 3871 1201 230 24 & \\ 
\hline
107  &  425 2553 6988 11317 11888 8388 3987 1245 240 25 & \\ 
\hline
108  &  405 2407 6518 10442 10851 7578 3571 1110 215 23 & \\ 
\hline
109  &  405 2427 6638 10751 11296 7969 3785 1181 228 24 & S30 \\ 
\hline
110  &  465 2904 8312 14152 15700 11734 5907 1940 384 38 & \\ 
\hline
111  &  464 2902 8323 14204 15795 11828 5960 1956 386 38 & \\ 
\hline
112  &  438 2690 7559 12608 13667 9952 4868 1552 300 30 & \\ 
\hline
113  &  445 2725 7617 12611 13546 9764 4728 1495 288 29 & \\ 
\hline
114  &  437 2669 7447 12319 13236 9556 4642 1475 286 29 & \\ 
\hline
115  &  411 2470 6776 11012 11617 8235 3933 1234 239 25 & \\ 
\hline
116  &  424 2574 7139 11737 12529 8983 4332 1367 264 27 & \\ 
\hline
117  &  419 2522 6922 11243 11842 8373 3985 1245 240 25 & \\ 
\hline
118  &  401 2387 6477 10398 10825 7570 3570 1110 215 23 & \\ 
\hline
119  &  405 2427 6638 10751 11296 7969 3785 1181 228 24 & S6 \\ 
\hline
120  &  464 2893 8261 14019 15483 11503 5746 1869 366 36 & \\ 
\hline
121  &  454 2806 7928 13283 14448 10543 5159 1641 315 31 & \\ 
\hline
122  &  451 2794 7928 13370 14676 10840 5387 1746 342 34 & \\ 
\hline
123  &  444 2736 7715 12915 14053 10273 5044 1613 312 31 & \\ 
\hline
124  &  466 2909 8318 14138 15644 11650 5837 1906 375 37 & \\ 
\hline
125  &  456 2815 7939 13271 14398 10480 5118 1627 313 31 & \\ 
\hline
126  &  423 2561 7078 11586 12303 8767 4199 1316 253 26 & \\ 
\hline
127  &  429 2580 7064 11429 11972 8402 3959 1221 232 24 & \\ 
\hline
128  &  431 2626 7309 12058 12915 9290 4494 1422 275 28 & \\ 
\hline
129  &  428 2602 7224 11883 12684 9087 4375 1377 265 27 & \\ 
\hline
130  &  443 2727 7679 12831 13927 10147 4960 1577 303 30 & \\ 
\hline
131  &  432 2637 7354 12152 13024 9356 4505 1412 269 27 & \\ 
\hline
132  &  451 2793 7920 13342 14620 10770 5331 1718 334 33 & \\ 
\hline
133  &  434 2632 7273 11879 12557 8883 4210 1301 246 25 & \\ 
\hline
134  &  452 2781 7813 13004 14042 10171 4944 1566 301 30 & \\ 
\hline
135  &  453 2808 7969 13433 14725 10847 5366 1727 335 33 & \\ 
\hline
136  &  451 2794 7928 13370 14676 10840 5387 1746 342 34 & \\ 
\hline
137  &  433 2646 7390 12236 13150 9482 4589 1448 278 28 & \\ 
\hline
138  &  442 2715 7629 12727 13808 10076 4948 1587 309 31 & \\ 
\hline
139  &  432 2633 7332 12104 12975 9341 4521 1430 276 28 & \\ 
\hline
140  &  423 2564 7096 11632 12368 8822 4227 1324 254 26 & \\ 
\hline
141  &  413 2483 6808 11043 11606 8177 3871 1201 230 24 & \\ 
\hline
142  &  427 2594 7196 11827 12614 9031 4347 1369 264 27 & \\ 
\hline
143  &  431 2622 7281 11973 12769 9135 4390 1379 265 27 & \\ 
\hline
144  &  431 2626 7309 12058 12915 9290 4494 1422 275 28 & \\ 
\hline
145  &  410 2459 6725 10881 11411 8029 3802 1183 228 24 & \\ 
\hline
146  &  428 2594 7176 11761 12514 8947 4307 1359 263 27 & \\ 
\hline
147  &  419 2522 6922 11243 11842 8373 3985 1245 240 25 & \\ 
\hline
148  &  451 2781 7840 13111 14243 10390 5089 1623 313 31 & \\ 
\hline
149  &  464 2900 8310 14168 15740 11778 5933 1948 385 38 & \\ 
\hline
150  &  446 2750 7757 12985 14123 10315 5058 1615 312 31 & \\ 
\hline
151  &  420 2541 7021 11496 12218 8719 4184 1314 253 26 & \\ 
\hline
152  &  441 2705 7584 12611 13622 9885 4823 1537 298 30 & \\ 
\hline
153  &  425 2575 7119 11651 12363 8799 4208 1317 253 26 & \\ 
\hline
154  &  448 2764 7801 13067 14223 10397 5102 1629 314 31 & \\ 
\hline
155  &  444 2737 7724 12949 14124 10363 5115 1647 321 32 & \\ 
\hline
156  &  452 2772 7753 12830 13755 9876 4750 1486 282 28 & \\ 
\hline
157  &  442 2706 7565 12529 13460 9696 4684 1473 281 28 & \\ 
\hline
158  &  441 2708 7602 12655 13676 9915 4821 1525 292 29 & \\ 
\hline
159  &  427 2596 7207 11850 12633 9026 4324 1350 257 26 & \\ 
\hline
160  &  452 2781 7813 13004 14042 10171 4944 1566 301 30 & \\ 
\hline
161  &  427 2586 7144 11681 12383 8806 4209 1317 253 26 & \\ 
\hline
162  &  400 2382 6467 10388 10820 7569 3570 1110 215 23 & \\ 
\hline
163  &  448 2764 7800 13061 14208 10377 5087 1623 313 31 & \\ 
\hline
164  &  470 2943 8444 14405 16000 11959 6011 1967 387 38 & \\ 
\hline
165  &  460 2857 8117 13684 14985 11014 5430 1740 336 33 & \\ 
\hline
166  &  418 2530 6996 11466 12198 8712 4183 1314 253 26 & \\ 
\hline
167  &  434 2640 7325 12025 12788 9108 4348 1353 257 26 & \\ 
\hline
168  &  425 2577 7132 11687 12418 8849 4235 1325 254 26 & \\ 
\hline
169  &  425 2581 7160 11772 12564 9004 4339 1368 264 27 & \\ 
\hline
170  &  430 2614 7255 11928 12724 9109 4382 1378 265 27 & \\ 
\hline
171  &  422 2557 7075 11597 12333 8801 4220 1323 254 26 & \\ 
\hline
172  &  411 2470 6772 10988 11556 8150 3863 1200 230 24 &  S7 \\ 
\hline
173  &  427 2586 7144 11681 12383 8806 4209 1317 253 26 & \\ 
\hline
174  &  400 2382 6467 10388 10820 7569 3570 1110 215 23 & \\ 
\hline
175  &  464 2898 8295 14119 15649 11673 5856 1913 376 37 & \\ 
\hline
176  &  442 2718 7644 12754 13822 10056 4911 1562 301 30 & \\ 
\hline
177  &  440 2698 7563 12576 13587 9864 4816 1536 298 30 & \\ 
\hline
178  &  423 2562 7083 11596 12313 8772 4200 1316 253 26 & \\ 
\hline
179  &  452 2781 7813 13004 14042 10171 4944 1566 301 30 & \\ 
\hline
\caption{Orbits of maximal prime cones for $\Flag_5$, the F-vectors of the corresponding polytopes, and combinatorially equivalent string polytopes resp. FFLV polytope.}\label{tab:f-vector5}
\end{longtable}

\end{center}

\section*{Algebraic invariants of the $\Flag_5$ string polytopes}\label{app:string5}

The  table below contains information on the $\Flag_5$ string polytopes and the FFLV polytope for $\rho$. It shows the reduced expressions underlying the string polytopes, whether the polytopes satisfy the weak Minkowski property, the weight vectors constructed in \S\ref{string:weight}, and whether the corresponding initial ideal is prime. The last column contains information on unimodular equivalences among these polytopes. If there is no information in this column this means that there is no unimodular equivalence between this polytope and any other polytope in the table. 
\begin{table}
\begin{center}
\begin{tabular}{| l|  l| l| l|  l| l| l| }
\hline
 &$\underline w_0$ & Normal & MP & Weight vector & Prime & Uni. Eq. \\
\hline
S1& 1213214321 & yes &yes &  \shortstack{$w_1=(0, 512, 384, 112, 0, 256, 96, 768, 608, 480, 0, 64, 320, 832, 15, 14,$\\$
 526, 398, 126, 12, 268, 108, 780, 620, 492, 0, 8, 72, 328, 840)$} & yes & \shortstack{S18, S26,\\ S29} \\ \hline
S2& 1213243212 & yes &yes &  \shortstack{$w_2=(0,512,384,98, 0,256,96, 768,608, 480, 0, 64,320, 832, 30,28,540, $\\$412, 123, 24,280,120,792, 
	632, 504, 0,16,80,336, 848)$} & yes & - \\ \hline
S3 & 1213432312 & yes &no &  \shortstack{$w_3= (0, 512, 384, 74, 0, 256, 72, 768, 584, 456, 0, 64, 320, 832, 58, 56, 568,$\\$
	440, 111, 48, 304, 108, 816, 620, 492, 0, 32, 96, 352, 864)$} & no & - \\ \hline
S4& 1214321432 & yes &no &  \shortstack{$w_4= (0, 512, 384, 56, 0, 256, 48, 768, 560, 432, 0, 32, 288, 800, 120, 112, 624, $\\$496, 63, 96, 
	352, 54, 864, 566, 438, 0, 64, 36, 292, 804)$} & no & - \\ \hline
S5& 1232124321 & yes &yes &  \shortstack{$w_5= (0, 512, 288, 224, 0, 256, 192, 768, 704, 432, 0, 128, 384, 896, 15, 14, 526,$\\$ 302, 238, 12,
     268, 204, 780, 716, 444, 0, 8, 136, 392, 904)$} & yes & - \\ \hline
S6& 1232143213 & yes &yes &  \shortstack{$w_6= (0, 512, 288, 224, 0, 256, 192, 768, 704, 420, 0, 128, 384, 896, 30, 28, 540,$\\$ 316, 252, 24,
      280, 216, 792, 728, 437, 0, 16, 144, 400, 912)$} & yes & - \\ \hline
S7& 1232432123 & yes &yes &  \shortstack{$w_7 = (0, 512, 260, 196, 0, 256, 192, 768, 704, 390, 0, 128, 384, 896, 60, 56, 568, $\\$310, 246, 48,
      304, 240, 816, 752, 423, 0, 32, 160, 416, 928)$} & yes & - \\ \hline       
S8& 1234321232 & yes &no &  \shortstack{$w_8= (0, 512, 264, 152, 0, 256, 144, 768, 656, 396, 0, 128, 384, 896, 120, 112, 624, $\\$364, 219, 96,
	352, 210, 864, 722, 462, 0, 64, 192, 448, 960)$} & no & - \\ \hline
S9& 1234321323 & yes &no &  \shortstack{$w_9= (0, 512, 264, 152, 0, 256, 144, 768, 656, 394, 0, 128, 384, 896, 120, 112, 624,$\\$ 362, 222, 96,
	352, 212, 864, 724, 459, 0, 64, 192, 448, 960)$}& no & - \\	 \hline
S10& 1243212432 & yes &no &  \shortstack{$w_{10}= (0, 512, 272, 112, 0, 256, 96, 768, 608, 344, 0, 64, 320, 832, 240, 224, 736, $\\$472, 119, 192,
      448, 102, 960, 614, 350, 0, 128, 68, 324, 836)$} & no & - \\ \hline           
      \end{tabular}
\end{center}
\end{table} 

\begin{table}
\begin{tabular}{| l|  l| l| l|  l| l| l| }
\hline
 &$\underline w_0$ & Normal& MP & weight vector & Prime & Uni. Eq. \\
\hline   
S11& 1243214323 & yes &no &  \shortstack{$w_{11} = (0, 512, 272, 112, 0, 256, 96, 768, 608, 338, 0, 64, 320, 832, 240, 224, 736, $\\$466, 126, 192,
      448, 108, 960, 620, 347, 0, 128, 72, 328, 840)$}  & no & - \\ \hline
S12& 1321324321 &yes & no &  \shortstack{$w_{12}= (0, 512, 192, 448, 0, 128, 384, 640, 896, 240, 0, 256, 160, 672, 15, 14, 526,$\\$206, 462, 12,
      140, 396, 652, 908, 252, 0, 8, 264, 168, 680)$} & no & - \\ \hline    
S13& 1321343231 & yes &no &  \shortstack{$w_{13}= (0, 512, 192, 448, 0, 128, 384, 640, 896, 228, 0, 256, 160, 672, 29, 28, 540,$\\$ 220, 476, 24,
      152, 408, 664, 920, 246, 0, 16, 272, 176, 688)$} & no & - \\ \hline
S14& 1321432143 & yes &no &  \shortstack{$w_{14}= (0, 512, 192, 448, 0, 128, 384, 640, 896, 216, 0, 256, 144, 656, 60, 56, 568,$\\$248, 504, 48,
      176, 432, 688, 944, 219, 0, 32, 288, 146, 658)$} & no & - \\ \hline
S15& 1323432123 & yes &no &  \shortstack{$w_{15}= (0, 512, 132, 388, 0, 128, 384, 640, 896, 198, 0, 256, 192, 704, 60, 56, 568,$\\$182, 438, 48,
      176, 432, 688, 944, 231, 0, 32, 288, 224, 736)$} & no & - \\ \hline

S16& 1324321243 &yes & no &  \shortstack{$w_{16} = (0, 512, 136, 392, 0, 128, 384, 640, 896, 172, 0, 256, 160, 672, 120, 112, 624,$\\$ 236, 492, 96,
      224, 480, 736, 992, 175, 0, 64, 320, 162, 674)$} & no & - \\ \hline
S17& 1343231243 & yes &no &  \shortstack{$w_{17}= (0, 512, 48, 304, 0, 32, 288, 544, 800, 60, 0, 256, 40, 552, 240, 224, 736, 188,$\\$444, 192,
      168, 424, 680, 936, 63, 0, 128, 384, 42, 554)$} & no & - \\ \hline
S18& 2123214321 & yes &yes &  \shortstack{$w_{18}= (0, 256, 768, 112, 0, 512, 96, 384, 352, 864, 0, 64, 576, 448, 15, 14, 270, 782,$\\$126, 12, 524,
      108, 396, 364, 876, 0, 8, 72, 584, 456)$} & yes & \shortstack{S1, S26, \\ S29}  \\ \hline
S19& 2123243212 & yes &yes &  \shortstack{$w_{19}= (0, 256, 768, 98, 0, 512, 96, 384, 352, 864, 0, 64, 576, 448, 30, 28, 284, 796,$\\$ 123, 24, 536,
      120, 408, 376, 888, 0, 16, 80, 592, 464)$} & yes & - \\ \hline
S20& 2123432132 & yes &no &  \shortstack{$w_{20}= (0, 256, 768, 76, 0, 512, 72, 384, 328, 840, 0, 64, 576, 448, 60, 56, 312, 824,$\\$111, 48, 560,
      106, 432, 362, 874, 0, 32, 96, 608, 480)$} & no & - \\ \hline
S21& 2132134321 &yes & yes &  \shortstack{$w_{21} = (0, 256, 768, 224, 0, 512, 192, 320, 448, 960, 0, 128, 640, 336, 15, 14, 270,$\\$782, 238, 12,
      524, 204, 332, 460, 972, 0, 8, 136, 648, 344)$} & yes & - \\ \hline
S22& 2132143214 &yes & yes &  \shortstack{$w_{22} = (0, 256, 768, 224, 0, 512, 192, 320, 448, 960, 0, 128, 640, 328, 30, 28, 284,$\\$ 796, 252, 24,
      536, 216, 344, 472, 984, 0, 16, 144, 656, 329)$} &yes & - \\ \hline
S23& 2132343212 & yes &yes &  \shortstack{$w_{23} = (0, 256, 768, 194, 0, 512, 192, 320, 448, 960, 0, 128, 640, 352, 30, 28, 284,$\\$796, 219, 24,
      536, 216, 344, 472, 984, 0, 16, 144, 656, 368)$} & yes & - \\ \hline
S24& 2132432124 & yes &yes &  \shortstack{$w_{24}= (0, 256, 768, 196, 0, 512, 192, 320, 448, 960, 0, 128, 640, 336, 60, 56, 312,$\\$824, 246, 48,
      560, 240, 368, 496, 1008, 0, 32, 160, 672, 337)$} & yes & - \\ \hline
S25& 2134321324 &yes & no &  \shortstack{$w_{25}= (0, 256, 768, 152, 0, 512, 144, 272, 400, 912, 0, 128, 640, 276, 120, 112,$\\$368, 880, 222, 96,
      608, 212, 340, 468, 980, 0, 64, 192, 704, 277)$} & no & - \\ \hline    
      
S26& 2321234321 &yes & yes &  \shortstack{$w_{26}= (0, 64, 576, 448, 0, 512, 384, 96, 352, 864, 0, 256, 768, 112, 15, 14, 78,$\\$590, 462, 12, 524,
      396, 108, 364, 876, 0, 8, 264, 776, 120)$} & yes & \shortstack{S1, S18, \\ S29} \\  \hline
S27& 2321243214 & yes &yes &  \shortstack{$w_{27}= (0, 64, 576, 448, 0, 512, 384, 96, 352, 864, 0, 256, 768, 104, 30, 28, 92,$\\$604, 476, 24, 536,
      408, 120, 376, 888, 0, 16, 272, 784, 105)$} & yes & - \\ \hline
S28& 2321432134 & yes &yes &  \shortstack{$w_{28}= (0, 64, 576, 448, 0, 512, 384, 72, 328, 840, 0, 256, 768, 74, 60, 56, 120,$\\$632, 504, 48, 560,
      432, 106, 362, 874, 0, 32, 288, 800, 75)$} & yes & - \\ \hline
S29& 2324321234 & yes &yes &  \shortstack{$w_{29}= (0, 8, 520, 392, 0, 512, 384, 12, 268, 780, 0, 256, 768, 14, 120, 112, 108,$\\$620, 492, 96, 608,
      480, 78, 334, 846, 0, 64, 320, 832, 15)$} & yes &  \shortstack{S1, S18,\\ S26}\\ \hline
S30& 2343212324 &yes & yes &  \shortstack{$w_{30}= (0, 16, 528, 304, 0, 512, 288, 24, 280, 792, 0, 256, 768, 28, 240, 224, 216,$\\$ 728, 438, 192,
      704, 420, 156, 412, 924, 0, 128, 384, 896, 29)$} & yes & - \\ \hline
S31& 2343213234 &yes & yes &  \shortstack{$w_{31}= (0, 16, 528, 304, 0, 512, 288, 20, 276, 788, 0, 256, 768, 22, 240, 224, 212,$\\$724, 444, 192,
      704, 424, 150, 406, 918, 0, 128, 384, 896, 23)$} & yes & - \\ \hline
FFLV & - & yes &yes &  \shortstack{$w^{reg}=(0,4,6,6,0,3,4,6,6,9,0,2,4,6,4,3,4,7,8,2,3,5,4,6,8,0,1,2,3,4)$\\
$w^{min}=(0,3,4,3,0,2,2,4,3,5,0,1,2,3,1,1,1,3,3,1,1,2,1,2,3,0,1,1,1,1)$}  &\shortstack{yes\\yes} &-\\
\hline
	  
    \end{tabular}\caption{String polytopes depending on $\underline w_0$ and the FFLV polytope for $\Flag_5$ and $\rho$, their normality, the weak Minkowski property, the weight vectors constructed in \S \ref{string:weight}, primeness of the binomial initial ideals, and unimodular equivalences among the polytopes.}\label{tab:stringweight5}
\end{table}

\end{document}